\documentclass[12pt]{amsart}

\allowdisplaybreaks

\usepackage{amsmath,amssymb,verbatim,graphicx,amsthm,color,url, enumerate, enumitem}
\usepackage[top=0.5in, bottom=0.5in, left=1.2in, right=1.2in]{geometry}
\usepackage{cite}
\usepackage{buczospectra}
\usepackage[colorlinks=true]{hyperref} 
\usepackage{ifthen}
\usepackage{tikz}
\usetikzlibrary{decorations.pathreplacing}
\usepackage{tikz-cd}
\numberwithin{equation}{section}
\usepackage{color}
\usepackage{fancyhdr}
\usepackage{datetime}
\allowdisplaybreaks[1]
\usepackage[sc]{mathpazo}
\usepackage[T1]{fontenc}
\usepackage{blindtext}
\usepackage{scrextend}
\addtokomafont{labelinglabel}{\sffamily}

\newcommand{\vertiii}[1]{{\left\vert\kern-0.25ex\left\vert\kern-0.25ex\left\vert #1
    \right\vert\kern-0.25ex\right\vert\kern-0.25ex\right\vert}}
\newcommand{\braciii}[1]{{\left[\kern-0.25ex\left[\kern-0.25ex\left[ #1
\right]\kern-0.25ex\right]\kern-0.25ex\right]}}
\newcommand{\norm}[1]{\left\lVert#1\right\rVert}

\newcommand{\RR}{\mathbb{R}}

\newcommand{\NN}{\mathbb{N}}

\newcommand{\ind}{\mathbf{1}}

\newcommand{\diam}{\text{diam}}

\newcommand{\PCC}{\mathrm{PCC}}

\theoremstyle{plain}

\newtheorem{theorem}{Theorem}[subsection]
\newtheorem*{theorem*}{Theorem}

\newtheorem{lemma}[theorem]{Lemma}

\newtheorem{claim}[theorem]{Claim}

\theoremstyle{definition}

\newtheorem{remark}[theorem]{Remark}
\newtheorem{example}[theorem]{Example}

\title{Generic Birkhoff Spectra}

\author{Zolt\'an Buczolich}
\address{Department of Analysis, ELTE E\"otv\"os Lor\'and\\
University, P\'azm\'any P\'eter S\'et\'any 1/c, 1117 Budapest, Hungary\\ORCID Id: 0000-0001-5481-8797}
\email{buczo@caesar.elte.hu}
\urladdr{http://buczo.web.elte.hu}

\author{Bal\'azs Maga}
\address{Department of Analysis, ELTE E\"otv\"os Lor\'and\\
University, P\'azm\'any P\'eter S\'et\'any 1/c, 1117 Budapest, Hungary}
\email{magab@caesar.elte.hu}
\urladdr{www.magab.web.elte.hu}

\author{Ryo Moore}
\address{Fatulcad de Matem\'aticas, Pontificia Universidad Cat\'olica de Chile, Santiago, Chile}
\email{rymoore@mat.uc.cl}
\urladdr{https://sites.google.com/view/ryomoore/}

\thanks{ZB was supported by the Hungarian National Research, Development and 
Innovation Office--NKFIH, Grant  124003.
}
\thanks{BM was supported by the \'UNKP-18-2 New 
National Excellence of the Hungarian Ministry of Human Capacities, and 
by the Hungarian National Research, Development and Innovation 
Office--NKFIH, Grant 124749.}
\thanks{RM was partially supported by CONICYT-FONDECYT Postdoctorado 3170279. 
}

\begin{document}
\maketitle

\begin{abstract}
Suppose that $\Omega = \{0, 1\}^\NN$ and  $\sss$ is the one-sided shift.
The Birkhoff spectrum 
$\ds S_{f}(\aaa)=\dim_{H}\Big \{ \ooo\in\OOO:\lim_{N \to \infty} \frac{1}{N} \sum_{n=1}^N f(\sigma^n \omega) = \alpha \Big  \},$
where $\dim_{H}$ is the Hausdorff dimension.
It is well-known that the support of $S_{f}(\aaa)$  is a bounded and closed interval
 $L_f = [\alpha_{f, \min}^*, \alpha_{f, \max}^*]$ 
 and $S_{f}(\aaa)$ on $L_{f}$ is concave and upper semicontinuous.
We are interested in possible shapes/properties of the spectrum, especially
for generic/typical $f\in C(\OOO)$ in the sense of Baire category. 
For a dense set in $C(\OOO)$ the spectrum is not continuous on $\RR$, though
for the generic $f\in C(\OOO)$ the spectrum is continuous on $\RR$, but has infinite one-sided derivatives at the endpoints of
$L_{f}$. We give an example of a function which has continuous $S_{f}$ on $\RR$, but with finite one-sided derivatives at the endpoints of $L_{f}$. The spectrum of this function
can be as close as possible to a "minimal spectrum".
We use that if  two functions $f$ and $g$ are close in  $C(\OOO)$ then $S_{f}$ and $S_{g}$ are close  on $L_{f}$ apart from neighborhoods of the endpoints.
\end{abstract}

\tableofcontents

\section{Introduction}

\subsection{Background}
Let $(X, \mathcal{F}, \mu, T)$ be a measure-preserving system. 
Birkhoff's Ergodic Theorem tells us that for $\mu$-a.e. $x \in X$ and $f \in L^1(\mu)$, the  limit\\ $\lim_{N \to \infty} \frac{1}{N} \sum_{n=1}^N f(T^nx)$ exists, and is a $T$-invariant function. 
Furthermore, if $T$ is ergodic with respect to $\mu$, the limit equals the constant $\int f \, d\mu$. 
For the ergodic case, if we let $E_f(\alpha):= \{x \in X: \lim_{N \to \infty} \frac{1}{N} \sum_{n=1}^N f(T^nx) = \alpha\}$ then $\mu(E_f(\alpha)) = 1$ if $\alpha = \int f \, d\mu$, and $0$ otherwise.

Now consider $(X, T)$ to be a topological dynamical system, and $f$ be a continuous function on $X$. 
Instead measuring the level-set $E_f(\alpha)$ by the ergodic measure $\mu$, one gets more interesting values by considering the Hausdorff dimension of the sets $E_{f}(\aaa)$ 
(including the irregular set $E_f' := \{x \in X: \lim_{N \to \infty} \frac{1}{N} \sum_{n=1}^N f(T^nx) $ does not exist.$ \}$).  For a given measure $\mu$ the Birkhoff Ergodic Theorem selects just one $\aaa$ and gives zero $\mmm$ measure to the other
sets $E_{f}(\aaa')$ for $\aaa'\neq\aaa$.
The function $S_f(\alpha) := \dim_H(E_f(\alpha))$ is called the \textit{Birkhoff spectrum} for $f$, and it will be the primary object that we study in this paper.

Such kind of study is referred to as a multifractal analysis. Multifractal analysis on Birkhoff averages has been initiated by Y. Pesin and H. Weiss \cite{PW01} for H\"older functions in the context of thermodynamic formalism. 
Birkhoff spectrum of continuous functions was studied by A.-H. Fan, D.-J. Feng, and J. Wu \cite{Fan_Feng_Wu_Recurrence}. In their study (which we will recall precisely in Theorem $\ref{theoremFFW}$), they have shown a variational formula between the dimension of the level set and the metric entropy.
They have also shown that $S_f(\alpha)$ is concave and upper semicontinuous (hence continuous by the nature of concave functions; see \cite[\S10]{Rockafellar}) on the interior of the set $\{\alpha \in \RR^d: E_f(\alpha) \neq \emptyset \}$, while remaining the question regarding the behavior of the spectrum at the boundary 
of its support  open. 

For other studies of the Birkhoff spectrum, we refer to, for instance, \cite{BS01}, \cite{TV03}, \cite{C10}, \cite{FLW02}, \cite{JJOP10}, \cite{O03}, and \cite{IJT17}. For more information on multifractal analysis (especially with its relationship to thermodynamic formalism), we refer to 
\cite{CaMa}, \cite{Rand} and to
a survey paper of V. Climenhaga \cite{C14}.

The main objective of this paper is to better understand the Birkhoff spectrum for generic continuous functions. 
We recall that given a complete metric space $(X, d)$, we say a set $A \subset X$ is \textit{generic} (or \textit{typical}) if $A$ is a complement of a set of first category (i.e. a countable union of nowhere dense sets). The Baire category theorem asserts that a generic set $A$ is dense in $X$. 
In our paper we will work with the full shift  $(\Omega, \sigma)$ on the alphabet $\{ 0,1 \}$ and consider Birkhoff averages of real-valued continuous functions $f \in C(\Omega, \RR)=C(\OOO)$. 
One of the main foci of this paper will be on the behavior of the spectrum of a generic continuous function at the boundary of the support of the spectrum. 
In case of one-dimensional range the support of the  spectrum of $f\in C(\OOO)$ is always a (possibly degenerate) closed interval $L_{f}$ 
and concave and upper semiconinuous functions are always continuous on
such intervals. However, it may happen that $S_{f}$, as a function defined on $\RR$ has a jump discontinuity at the endpoints of $S_{f}$. 
Such functions were called degenerate by J. Schmeling in \cite{Schmeling}.
We will show that for the generic $f\in C(\OOO)$ the spectrum is continuous, with infinite one-sided derivatives at the endpoints of
$L_{f}$.  Continuity of the spectrum for the generic H\"older function was proved by Schmeling in  \cite{Schmeling}. In fact, this combined with results in 
\cite{Morris} and \cite{Fan_Feng_Wu_Recurrence}
imply the continuity of the spectrum for the generic continuous function in our setting.
In this paper we give a direct proof of this fact.

\subsection{Summary of the main results, organization of the paper}
Let $\Omega = \{0, 1\}^\NN$, and $\sigma$ be the shift map. 
We assume that $(\Omega, \sigma)$ is the full shift. 
The space of real-valued continuous functions on $\Omega$ (denoted $C(\Omega)$) is equipped with the usual supremum norm. 
We denote by $\alpha_{f, \max}$ (resp. $\alpha_{f, \min}$)  the maximum (resp. minimum) value of $f\in C(\OOO)$. 
The level-sets  of the Birkhoff averages are 
\begin{equation}\label{*defefa}
E_f(\alpha) :=\Big  \{\omega \in \Omega: \lim_{N \to \infty} \frac{1}{N} \sum_{n=1}^N f(\sigma^n \omega) = \alpha \Big \}.
\end{equation}
Let $\alpha_{f, \max}^* := \sup\{\alpha \in \RR : E_f(\alpha) \neq \emptyset \}$, and $\alpha_{f, \min}^*:= \inf\{\alpha \in \RR : E_f(\alpha) \neq \emptyset \}$. 
We also put $L_f = [\alpha_{f, \min}^*, \alpha_{f, \max}^*]$.
The Birkhoff spectrum is defined as $
S_f(\alpha) := 
\dim_H E_f(\alpha)
$, keeping in mind that  the empty set has Hausdorff dimension zero $S_{f}$ is defined on $\RR$. 
Results on concavity of $S_{f}$ on $L_{f}$ and Birkhoff's Ergodic Theorem imply  that $L_{f}$  is the support 
of $S_f$. It is known, \cite{Fan_Feng_Wu_Recurrence}, that $S_{f}$ is apart from being concave on the closed interval $L_f$ is also upper semicontinuous and hence it is continuous on the same closed interval. 
Often, for ease of terminology, we will mention the endpoints of the support of the spectrum as the endpoints of the spectrum.

In Section \ref{*secpreli} after introducing some notation we give some simple examples and recall one of the main results of \cite{Fan_Feng_Wu_Recurrence}.  

Next we discuss some tools used later.
First, we show that given a continuous function $f$, any continuous function that is sufficiently close to $f$ would have its Birkhoff spectrum also close to $S_f$ on $L_f$ except for a neighborhood of the endpoints of the spectrum. 
This will be proven in Theorem $\ref{Thm:NormCont}$.

In Subsection \ref{*secpccc} we prove some results about piecewise constant continuous (or simply $\PCC$)
functions, that is about functions which depend on finitely many coordinates. Among other
results we show that for such functions $f$ there is always a periodic $\ooo$ in $E_{f}(\aaa_{f,\max}^{*})$.

The next two results will concern the continuity of a Birkhoff spectrum. 
Given $f \in C(\Omega)$, we say that the spectrum $S_f$ is \textit{continuous} if it is continuous on $\RR$, and \textit{discontinuous} otherwise. 
Equivalently, $S_f$ is continuous when $S_f(\alpha_{f, \min}^*) = S_f(\alpha_{f, \max}^*) = 0$. 
We will first show that continuous, in fact $\PCC$ functions with discontinuous spectrum are dense in $C(\Omega)$ (Theorem $\ref{*ppropdisc}$).
 On the other hand, we give a direct proof of the fact that generic continuous functions have continuous spectrum (Theorem $\ref{contspectrum}$). 
In  \cite[\S5, Item (2)]{Fan_Feng_Wu_Recurrence} a question was raised about
continuity of the spectrum at the boundary of its support. In the one-dimensional case, as we mentioned
the answer is obvious if we consider the restriction of $S_{f}$ onto $L_{f}$, however
there might be discontinuity from the exterior side of $L_{f}$.

In Subsection \ref{SS-gensp} we show that for a dense open subset of $C(\OOO)$ the support of the  spectrum is in the interior of $[\aaa_{f,\min},\aaa_{f,\max}]$.

It is mentioned in the introduction of \cite{Fan_Feng_Wu_Recurrence} that
even for H\"older regular functions discussions of $S_{f}(\aaa)$  for boundary points 
of $L_{f}$ are scarce, which is actually a  subtle problem.

In the remainder of our paper, in Section \ref{Sonesderiv} we will discuss one-sided derivatives of a Birkhoff spectrum at the endpoints/boundary points of the spectrum. 
Given $\varphi: \RR \to \RR$, we denote by $\partial^{-}\varphi(\alpha)$ the left-hand derivative of $\varphi$ at $\alpha$ (if the value exists). Similarly, $\partial^{+}\varphi(\alpha)$ denotes the right-hand derivative. 
We will show that the spectrum of a generic continuous function $f$ has infinite one-sided derivatives at the endpoints of $L_f$, i.e. $\partial^+ f(\alpha_{f, \min}^*) = \infty$, and $\partial^- f(\alpha_{f, \max}^*) = -\infty$ (Theorem $\ref{typicalderivative}$). We construct a continuous function with continuous spectrum for which the one-sided derivatives at the endpoints are finite (Theorem $\ref{finitederivatives}$).
This function will also have a very small spectrum. By concavity of the spectrum on its support there is always a triangle which should be under the graph of the spectrum.
Our example will provide an example when the spectrum is very close to this lower estimate. 
In \cite{TV03} Takens and Verbitsky calculated the spectrum of the Manneville-Pomeau map.
It has a Birkhoff spectrum with a finite one-sided derivative at one of the endpoints.

It is not that obvious that functions with finite one-sided derivatives at the endpoints of the spectrum exist since for some well-known examples of functions with continuous spectrum, like the one discussed in Example \ref{examplecontspect} we have $\partial^+ f(\alpha_{f, \min}^*) = \infty$, and $\partial^- f(\alpha_{f, \max}^*) = -\infty$, however this function does not have a ``generic spectrum"
since $\alpha_{f, \min}^*$ equals  $\alpha_{f, \min}$ and $\alpha_{f, \max}^{*}$
equals $\alpha_{f, \max}$. As we mentioned earlier for the generic continuous  functions we always have $\alpha_{f, \min}<\alpha_{f, \min}^*<\alpha_{f, \max}^*<\alpha_{f, \max}$ see Theorem
\ref{*typsup}. In Theorem \ref{pccinfiniteder} we prove that for $\PCC$ functions $f$
with continuous spectrum we always have $\partial^+ f(\alpha_{f, \min}^*) = \infty$, and $\partial^- f(\alpha_{f, \max}^*) = -\infty$. This illustrates that for the proof of
Theorem $\ref{finitederivatives}$ one needs to use a more involved construction
than a $\PCC$ function.


\section{Preliminaries}\label{*secpreli}
\subsection{Notation and terminology}
Let $\Omega = \{0, 1\}^\NN$, and $\sigma$ be the shift map. 

We introduce the usual metric $d$ on $\Omega$ defined by
\[d(\omega,\omega')=\sum_{k=1}^{\infty}\frac{|\omega_k-\omega'_k|}{2^k},\]
where $\omega_k$ (resp. $\omega_k'$) denotes the coordinates/entries of $\omega$ (resp. $\omega'$).
If $k\in\NN\cup \{\oo\}$ and $A$ is a finite string of $0$s and $1$s then $A^{k}$ denotes the $k$-fold concatenation of $A$ and $[A]$ denotes the cylinder set $\ds \{ \ooo: A\ooo' ,\  \ooo'\in \OOO \}$.
Given $k,l\in\NN$ and $\ooo=(\ooo_{1}\ooo_{2}...)\in\OOO$ we put $\ooo|k=\ooo_{1}...\ooo_{k}$ and  $(\omega)_k^l := \omega_k\omega_{k+1}\ldots\omega_{l-1}\omega_l$,
if $k\leq 0$ then $\ooo|k$ is the empty string and analogously if $k>l$ then 
$(\omega)_k^l$ is also the empty string. The "conjugate" $\overline{\ooo}$ is the string which we obtain from $\ooo$ by swapping $0$s and $1$s, that is $\overline{\ooo}_{k}=1-\ooo_{k}$ for all $k$.

The $s$-dimensional Hausdorff measure of $A\sse \OOO$ is denoted by $\cah^{s}(A)$
and recall that $\cah^{s}(A)=\lim_{\ddd\to 0+}\cah^{s}_{\ddd}(A)$ where
$\cah^{s}_{\ddd}(A)=\inf \{\sum_{i}(\diam\ U_{i})^{s} :$ where $A\sse \cup_{i}U_{i}$ and $\diam\ U_{i}<\ddd \}.$ The Hausdorff dimension of $A\sse \OOO$ is $\dim_H A=\inf\{s: \cah^{s}(A)=0\}$. From this definition, it is a standard exercise to show that $\dim_H \Omega = 1$. 

The complement of a set $A$ is denoted by $A^{c}$.

Let $\PCC^k(\Omega)$ be the set of those piecewise constant continuous functions in $C(\Omega)$,  which
depend only on cylinders of length/depth $k$.
While the set of piecewise constant continuous functions in $C(\Omega)$, is denoted by
$\PCC(\OOO)$. 
Obviously $\PCC(\OOO)=\cup_{k}\PCC^{k}(\OOO)$.

The $(1/2, 1/2)$-Bernoulli measure, the ``Lebesgue measure" on $\Omega$ is denoted by $\lambda$. 
In case we write $\int f$ for an $f:\OOO\to\RR$ we always mean $\int_{\OOO}fd\lll$.

We denote  by $C_0(\Omega)$  the set of continuous functions for which $\int f  = 0$, and $\PCC_0^k(\Omega)=\PCC^{k}(\OOO)\cap C_0(\Omega).$ 

Given $f \in C(\Omega)$, we denote $\norm{f} = \sup_{\omega \in \Omega}|f(\omega)|$, and for any $\delta > 0$, $B(f, \delta) = \{g \in C(\Omega): \norm{f-g} < \delta \}$. 

Recall \eqref{*defefa} and the subsequent definitions of  $E_f(\alpha)$,
$S_{f}(\aaa)$.
We remark that our definition of $S_f(\alpha)$ is a bit different from the usual notation  in multifractal analysis, since quite often 
$S_{f}(\aaa)$ is defined to be $-\oo$ when $E_{f}(\aaa)$ is empty.
 
As previously defined, we set $\alpha_{f, \max}^* = \sup\{\alpha \in \RR : E_f(\alpha) \neq \emptyset \}$, where $\alpha_{f, \min}^* = \inf\{\alpha \in \RR : E_f(\alpha) \neq \emptyset \}$. 
In general we have $\alpha_{f, \min} \leq \alpha_{f, \min}^* \leq \alpha_{f, \max}^* \leq \alpha_{f, \max}$, and it is possible for the strict inequalities to hold (including the first and the third inequality), as we will see in an example (cf. 
Example $\ref{*propdisc}$). In fact, as Theorem \ref{*typsup} shows this property is true for the generic continuous functions as well.

The $\sss$-invariant Borel probability measures are denoted by $\mathcal{M}_\sigma$. 
By Birkhoff's Ergodic Theorem, we know that $\lambda(E_f(\int f)) = 1$. Furthermore, if $\{C_i\}_{i=1}^\infty$ are cylinders in $\OOO$  of length at least $k \in \NN$ and $E_f(\int f) \subset \bigcup_{i=1}^\infty C_i$ then
$$1 = \lambda\left(E_f\left(\int f\right)\right) \leq  \sum_{i=1}^\infty \lambda(C_i) = \sum_{i=1}^\infty \diam (C_i),$$
which implies that $1 \leq \mathcal{H}_{2^{-k}}(E_f(\int f)) \leq \mathcal{H}_{2^{-k}}(\Omega)$ for any $k \in \NN$, and thus
  $S_{f}(\int f d\lll)=1$. 
Given $f\in C(\OOO)$ and $\aaa\in\RR$ we will also use the following subsets of $\cam_{\sss}$
\begin{equation}\label{*defff}
\mathcal{F}_f(\alpha) :=\Big  \{\mu \in \mathcal{M}_\sigma : \int f \, d\mu = \alpha \Big \}. 
\end{equation}

\subsection{Examples}
We present a few examples of Birkhoff spectra of certain $PCC(\Omega)$ functions. 
We will first provide an example for a function with continuous spectrum.

\begin{example} \label{examplecontspect}
Let $f\in C(\Omega)$ be the function given by $f(\omega)=1$ if $\omega_1=1$ and $f(\omega)=0$ if $\omega_1=0$. 
Then for any $\alpha\in (0,1)$ we have 
\[S_f(\alpha)=-\frac{\alpha\log(\alpha)+(1-\alpha)\log(1-\alpha)}{\log 2},\]
if $\aaa\not \in (0,1)$ then $S_{f}(\aaa)=0$.
In particular, $f$ has continuous spectrum, as $\alpha_{f,\min}^*=0$, $\alpha_{f,\max}^*=1$, and furthermore, $\partial^+S_{f}(\alpha_{f, \min}^*) = +\infty$ and $\partial^-S_{f}(\alpha_{f, \max}^*) = -\infty$. 
\end{example}

\begin{proof}[Verification of the properties of Example \ref{examplecontspect}]
We will prove two inequalities using suitably defined H\"older functions and the result of \cite{Eggleston}. First, let us consider the function $h_1:\Omega\to[0,1]$ defined by
\[h_1(\omega)=\sum_{i=1}^{\infty}\frac{\omega_i}{2^i}.\]
That is, $h_1$ takes a 0-1 sequence to the number with the corresponding binary expansion. 
We claim that $h_1$ is a Lipschitz function in fact. 
Indeed, if $\omega'$ differs from $\omega$ in its $n$th coordinate, but not before that point, then $d(\omega,\omega')\geq{2^{-n}}$, while $|h_1(\omega)-h_1(\omega')|\leq{2^{-n+1}}$, hence $h_1$ has Lipschitz constant 2. Moreover, $h_1(E_f(\alpha))$ equals the set of numbers in $[0,1]$ having a binary expansion in which the density of $1$s equals $\alpha$. 
Thus due to \cite{Eggleston}, the dimension of $h_1(E_f(\alpha))$ is given by the formula in the statement of the lemma, yielding
\[S_f(\alpha)\geq-\frac{\alpha\log(\alpha)+(1-\alpha)\log(1-\alpha)}{\log 2},\]
as $h_1$ is Lipschitz.

Concerning the other inequality, define $h_2:C\to\Omega$ for the triadic Cantor set $C\subset[0,1]$: if the triadic expansion of $x\in C$ is
\[x=\sum_{i=1}^{\infty}\frac{x_i}{3^i},\]
then let $\omega=h_2(x)$ have coordinates $\frac{x_1}{2},\frac{x_2}{2},...$.
That is, $h_2$ is a one-to-one mapping between $\Omega$ and $C$. 
Now if $x$ differs from $x'$ in its $n$th coordinate, but not before that point, then $|x-x'|\geq {3^{-n}}$. 
On the other hand, $d(h_2(\omega),h_2(\omega'))\leq 2^{-n+1}$. 
It quickly yields that $h_2$ is a H\"older function with exponent $\frac{\log 2}{\log 3}$. 
Moreover, $h_2^{-1}(E_f(\alpha))$ is the set of numbers in $[0,1]$ having a ternary expansion with no $1$s, in which the density of $2$s is $\alpha$ and the density of $0$s is $1-\alpha$. 
Hence $h_2^{-1}(E_f(\alpha))$ is contained by the set of numbers in $[0,1]$ having a ternary expansion in which the density of $2$s is $\alpha$ and the density of $0$s is $1-\alpha$. 
Thus due to \cite{Eggleston}, the dimension of $h_2^{-1}(E_f(\alpha))$ is at most
\[-\frac{\alpha\log(\alpha)+(1-\alpha)\log(1-\alpha)}{\log 3}.\]
Hence as $h_2$ is $\frac{\log 2}{\log 3}$-H\"older, we obtain an upper estimate for $S_f(\alpha)$, that is the dimension of $E_f(\alpha)$, notably
\[S_f(\alpha)\leq-\frac{\alpha\log(\alpha)+(1-\alpha)\log(1-\alpha)}{\log 2}.\]
This shows that the desired equality holds, and the remaining claims clearly follow.
\end{proof}
\begin{remark}
	One can use the fact that $S_f$ is the Legendre transform of the topological pressure function $P(tf)$ to obtain a less direct argument that verifies the formula in Example \ref{examplecontspect}.
\end{remark}

Next, we will see examples of continuous functions with discontinuous spectra.

\begin{example}\label{cob_ex}
If $f$ is a constant function, i.e. $f \equiv C \in \RR$, then $S_f(C) =1$ and $S_{f}(\aaa)=0$ otherwise. 
The same is true if $f$ is cohomologous to a constant, i.e. there exists $g \in C(\Omega)$ for which $f = C  + g - g \circ \sigma$ (we recall that if $C$ is zero, $f$ is called a coboundary). 
\end{example}

Finally, we give an example where $\alpha_{f,\min}<\alpha_{f,\min}^*<\alpha_{f,\max}^*<\alpha_{f,\max}$ (that is, strict inequalities are satisfied), and the Birkhoff spectrum is discontinuous.

\begin{example}\label{*propdisc}
There exists
$f\in \PCC_0^{3}(\Omega)$ satisfying $\alpha_{f,\min}<\alpha_{f,\min}^*<\alpha_{f,\max}^*<\alpha_{f,\max}$ and $S_f(\alpha_{f,\min}^*),S_f(\alpha_{f,\max}^*)>0$.
\end{example}

\begin{proof}
As $f\in\PCC_0^3(\Omega)$ we can define it by giving its values on $3$-cylinders by abusing a bit the notation for $f$. 
We define $f$ by $f([000])=f([010])=-2$, $f([001])=-3$, $f([100])=-1$, and $f(\overline{\ooo})=-f(\ooo)$. 
Then we clearly have $\alpha_{f,\min}=-3$ while $\alpha_{f,\max}=3$.

Now we claim $\alpha_{f,\min}^*=-2$, while $\alpha_{f,\max}^*=2$, which would yield the inequalities $\alpha_{f,\min}<\alpha_{f,\min}^*<\alpha_{f,\max}^*<\alpha_{f,\max}$. 
Due to symmetry reasons, it suffices to verify $\alpha_{f,\min}^*=-2$. 
To this end, consider an arbitrary $\omega\in\Omega$. 
Now we are interested in the averages $\frac{1}{N} \sum_{n=1}^{N} f(\sigma^n \omega)$. 
In the sequence $f(\sigma^n \omega)$ each value is at least -2, except for the cases when the first three coordinates of $\sigma^n \omega$ are $001$. 
However, in this case the first three coordinates of $\sigma^{n+2}\omega$ contain at least two $1$s, or they are $100$. 
In either case, $f(\sigma^{n+2} \omega)\geq{-1}$. 
This argument shows that in the sum $\sum_{n=1}^{N} f(\sigma^n \omega)$ the summands with value -3 can be paired with summands with value at least -1, except for possibly the last one, whose pair does not appear in the sum. 
Besides that, all the other summands have value at least -2. 
Consequently, the average $\frac{1}{N} \sum_{n=1}^{N} f(\sigma^n \omega)\geq -2-\frac{3}{N}$, hence the limit is at least -2, verifying $\alpha_{f,\min}^*\geq -2$. 
For the other inequality, we may simply consider the identically 0 sequence, hence $\alpha_{f,\min}^*=-2$. 
It proves the claim of this paragraph.

It remains to show that $S_f(\alpha_{f,\min}^*),S_f(\alpha_{f,\max}^*)>0$. 
Due to symmetry reasons, these quantities are clearly equal, hence $S_f(\alpha_{f,\min}^*)>0$ would be sufficient. 
Consider the following subset of $\Omega$:
\begin{displaymath}
B=\{\omega\in\Omega: \omega_k=0 \text{ for } k\equiv{1,2} \mod 3\}.
\end{displaymath}
Then for any $\omega\in B$ and $n$ we have that at least two of the first three coordinates of $\sigma^n \omega$ equals 0. Consequently, $f(\sigma^n \omega)<0$. 
Moreover, similarly to the previous argument we find that the in the sum $\sum_{n=1}^{N} f(\sigma^n \omega)$ the summands with value -3 can be paired with summands with value -1, except for possibly the last one. 
All the other summands have value -2. Hence we find
\begin{displaymath}
-2-\frac{1}{N}\leq\frac{1}{N} \sum_{n=1}^{N} f(\sigma^n \omega)\leq -2.
\end{displaymath}
It proves that $B\subset E_f(-2)$, hence $\dim_H B>0$ would conclude the proof. 
However, this dimension can be calculated explicitly as $B$ is a self-similar set, which equals the disjoint union of its $2$ similar images, where the similarities have ratio $\frac{1}{8}$. 
Thus $\dim_H B= \frac{\log 2}{\log 8}=\frac{1}{3}$ by Hutchinson's Theorem \cite{Hutch}.
\end{proof}

\subsection{Variational formula}
The following result was obtained by Fan, Feng, and Wu. 
We present this result in the context of the full-shift on an alphabet of two symbols $(\Omega, \sigma)$ (in \cite{Fan_Feng_Wu_Recurrence}, they proved the result for a topologically mixing subshift of finite type). 

\begin{theorem}[{\cite[Theorem A]{Fan_Feng_Wu_Recurrence}}]\label{theoremFFW} Suppose that $f: \Omega \to \RR^d$ is a continuous function. 
We denote $L_f := \{\alpha \in \RR^{d}: \alpha = \lim_{N \to \infty} \frac{1}{N} \sum_{n=1}^N f(\sigma^n \omega) \text{ for some } \omega \in \OOO \}$. 
There exists a concave and upper semi-continuous function $\Lambda_f$ such that for any $\alpha\in L_f$
\[S_f(\alpha) := \dim_H(E_f(\alpha)) = \Lambda_f(\alpha), \]
and
\[ \Lambda_f(\alpha) = \max_{\mu \in \mathcal{F}_f(\alpha)}\frac{ h_\mu }{\log 2} \]
where $h_\mu$ is the metric entropy of $\mu$, and $\mathcal{F}_f(\alpha)$
can be defined analogously to \eqref{*defff}.

\end{theorem}
The function $\Lambda_f(\alpha)$ is defined in the same paper \cite[Proposition 5]{Fan_Feng_Wu_Recurrence} using the cardinality of the cylinders of large length that contain at least one point $\omega$ for which the Birkhoff average of $f$ of that length is close to $\alpha$. 
It was later shown that the quantity $\Lambda_f(\alpha)$ indeed agrees with $S_f(\alpha)$ for all $\alpha \in L_f$ \cite[Proposition 6]{Fan_Feng_Wu_Recurrence}.

\section{Tools}

\subsection{Norm Continuity Theorem}\label{SS-NCT}
We first prove that two Birkhoff spectra of two continuous functions are close (except near the endpoints) if those two functions are close in the supremum norm.

\begin{theorem}[Norm continuity theorem]\label{Thm:NormCont}
Let $f \in C(\Omega)$ for which $\alpha_{f, \min}^* < \alpha_{f, \max}^*$,
and $\varepsilon \in (0, {(\alpha_{f, \max}^* - \alpha_{f, \min}^*)}/{2})$ be given. 
Then there exists $\delta > 0$ such that for any $g \in B(f, \delta)$, we have $|S_f(\alpha) - S_g(\alpha)| <\varepsilon$ for all $\alpha \in 
(\alpha_{f, \min}^*+\eee,\alpha_{f, \max}^*-\eee)$.
\end{theorem}

\begin{remark}
	We will later learn that the generic continuous function satisfies the hypothesis of this theorem; see Theorem \ref{*typsup}.
\end{remark}

If one considers $f,g\in C(\OOO)$ with continuous spectrum then the above theorem can be used to show that for given $\eee>0$ one can find $\ddd>0$ such that
$\|f-g\|<\ddd$ implies that $\|S_{f}-S_{g}\|<\eee$. On the other hand, if $f$ has discontinuous spectrum, say $S_{f}(\aaa_{f,\max}^{*})>0$ then the density of functions
with continuous spectrum (Theorem \ref{contspectrum}) and Remark \ref{*remsup}
imply that arbitrary close to $f$ one can find functions $g$ such that
$\|S_{f}-S_{g}\|> S_{f}(\aaa_{f,\max}^{*})/2$.

To proceed, we first prove the following lemma.


\begin{lemma}\label{Lemma:ptEstimate}
Let $\varepsilon > 0$ be given. 
Suppose that $f \in C(\Omega)$, and $\alpha \in [\alpha_{f, \min}^*,\alpha_{f, \max}^*]$. 
Then for any $g \in {C}(\Omega)$ such that $\norm{f - g} <\varepsilon$, there exists $\alpha' \in (\alpha-\varepsilon, \alpha+\varepsilon)$ for which $S_g(\alpha') \geq S_f(\alpha)$. 
If $S_f(\alpha)=0$, but $E_{f}(\aaa)\not=\ess$ then $E_{g}(\aaa')\not=\ess$.
\end{lemma}

\begin{remark}\label{*remsup}
This implies that if $\|f-g\|<\eee$ then $|\aaa_{f,\max}^{*}-\aaa_{g,\max}^{*}|<\eee$ and $|\aaa_{f,\min}^{*}-\aaa_{g,\min}^{*}|<\eee$.
\end{remark}

\begin{proof}
Recall the definition of  $\caf_{f}(\aaa)$ from \eqref{*defff}.
By 
Theorem \ref{theoremFFW} there exists $\mu_0 \in \mathcal{F}_f(\alpha)$ for which 
\[S_f(\alpha) = \frac{h_{\mu_0}}{\log 2} = \frac{\max_{{\mu} \in \mathcal{F}_f(\alpha)} h_\mu}{\log 2} \, . \]

Set $\alpha' = \int g \, d\mu_0$. 
Since $\norm{f - g} <\varepsilon$, we have $\alpha' \in (\alpha -\varepsilon, \alpha +\varepsilon)$. For $\aaa\in [\alpha_{f, \min}^*,\alpha_{f, \max}^*]$ we have $E_{f}(\aaa)\not=\ess$.
If $S_{f}(\aaa)=0$, then $\aaa \in \{\alpha_{f, \min}^*, \alpha_{f, \max}^*\}$. Consider the map $f_*: \mathcal{M}_\sigma \to L_f$ for which that $f_*(\mu) = \int f \, d\mu$. Since the map $f_*$ is affine and continuous, $\mu_0$ must be one of the extremal points of the convex set $\mathcal{M}_\sigma$. This implies that $\mu_0$ is ergodic, so we may apply Birkhoff's Ergodic Theorem to show that for $\mmm_{0}$
almost every $\ooo$ we have 
$\lim_{N\to\oo}\frac{1}{N}\sum_{n=1}^{N}g(\sss^{n}\ooo)=\aaa'$
and hence $E_{g}(\aaa')\not=\ess$.

Hence, from now on we can suppose that $S_{f}(\aaa)>0$.
In that case since $\mmm_{0}\in \mathcal{F}_g(\alpha')$ by Theorem \ref{theoremFFW} we obtain that
\[S_g(\alpha')=\frac{\max_{{\mu} \in \mathcal{F}_g(\alpha')} h_\mu}{\log 2} \geq \frac{h_{\mu_0}}{\log 2}= S_f(\alpha).\]
\end{proof}

Using this lemma, we will prove the theorem by using concavity of the spectrum.

\begin{proof}[Proof of Theorem $\ref{Thm:NormCont}$]
For some $L \in \NN$, we consider a partition
\[\alpha_{f, \min}^
*
= \alpha_1 < \alpha_2 < \cdots < \alpha_L = \alpha_{ 
f, \max}^* \]
for which for every $i = 1, 2, \ldots, L-1$, $|\alpha_{i+1} - \alpha_i|<\eee/4$ is small enough such that for every $t \in [0, 1]$, we have
\[(1-t)S(\alpha_{i}) + tS(\alpha_{i+1}) > S((1-t)\alpha_i + t\alpha_{i+1}) -\varepsilon/2. \]
For each $\alpha_i$, we choose a positive number $\delta(\alpha_i)<\eee/8$ as follows:
 For any $\alpha_i' \in (\alpha_i - \delta(\alpha_i), \alpha_i + \delta(\alpha_i))$, and $\beta_i' \geq S_f(\alpha_i)$, the line segments connecting the points $(\alpha_i', \beta_i')$ and $(\alpha_{i+1}', \beta_{i+1}')$ 
 are above the graph of $S_f(\alpha) -\varepsilon$ for  $i=2,...,L-2$. 
 We can also suppose that the intervals $(\alpha_i - \delta(\alpha_i), \alpha_i + \delta(\alpha_i))$ are disjoint.
Then we set
\[\delta = \min \{\varepsilon /8, \delta(\alpha_1), \delta(\alpha_2), \ldots, \delta(\alpha_L)\}. \]

 We apply Lemma $\ref{Lemma:ptEstimate}$ with $\eee=\ddd$ to show that there exists $\alpha_i' \in (\alpha_i -\ddd, \alpha_i +\ddd)\sse
 (\alpha_i -\ddd(\aaa_{i}), \alpha_i +\ddd(\aaa_{i}))$ such that $S_g(\alpha_i') \geq S_f(\alpha_i)$ for $i=1,...,L-1$. 
 Since $|\aaa_{1}'-\aaa_{f,\min}^{*}|= |\aaa_{1}'-\aaa_{1}|<\eee/8$
 and $|\aaa_{L}'-\aaa_{f,\max}^{*}|= |\aaa_{L}'-\aaa_{L}|<\eee/8$
by using the concavity of $S_g \, $ one can show that
$S_g(\alpha) > S_f(\alpha) -\varepsilon$
for all $\alpha \in (\alpha_{f, \min}^*+(\eee/2),\alpha_{f, \max}^*-(\eee/2))$. 
By reversing the roles of $f$ and $g$, by an analogous argument we can conclude that
$S_f(\alpha) > S_g(\alpha) -\varepsilon$ for all $\aaa\in (\alpha_{g, \min}^*+(\eee/2),\alpha_{g, \max}^*-(\eee/2))$. Using Remark \ref{*remsup} we can conclude the proof.
\end{proof}

\subsection{Piecewise constant (PCC) functions}\label{*secpccc}

We start with a lemma in which we show that $\alpha_{f, \max}^*$ is a uniform upper bound of the limit of the Birkhoff averages of any $f \in PCC^k$.


\begin{lemma} \label{uniformity}
Assume $f\in\PCC^k(\Omega)$ and $\varepsilon>0$. 
Then there exists $N_0$ such that for any $N \geq N_0$, for any $\omega \in \Omega$, we have
\begin{equation}\label{avgN4}
\frac{1}{N} \sum_{n=1}^{N} f(\sigma^n \omega) \leq 
\alpha_{f, \max}^* +\varepsilon,
\end{equation}
which implies that
\begin{equation}\label{*avgN4}
\limsup_{N\to\oo} \frac{1}{N} \sum_{n=1}^{N} f(\sigma^n \omega) \leq 
\alpha_{f, \max}^*\text{ uniformly for any }\ooo\in\OOO .
\end{equation}
\end{lemma}

\begin{proof}
Choose $N_0$ such that for any $N>N_0$
\begin{equation}\label{avgestimate}
\frac{-k\norm{f}+N(\alpha_{f, \max}^* +\varepsilon)}{N+k}>\alpha_{f, \max}^*+\frac{\varepsilon}{2}.
\end{equation}
We claim that this $N_0$ satisfies the statement of the lemma. 
Proceeding towards a contradiction, assume the existence of a configuration $\omega$ and $N>N_0$ which refutes this claim, that is
\begin{equation}\label{contradict}
\frac{1}{N} \sum_{n=1}^{N} f(\sigma^n \omega) >
\alpha_{f, \max}^* +\varepsilon.
\end{equation}
Our goal is to construct $\ooo'\in \OOO$, periodic by $N+k$
which will satisfy 
\begin{equation}\label{*contrav}
\sum_{n=1}^{N} f(\sigma^n \omega')=\sum_{n=1}^{N} f(\sigma^n \omega)>N(\alpha_{f, \max}^* +\varepsilon),
\end{equation}
 and this will contradict the definition of $\aaa_{f,\max}^{*}$ as we will see in \eqref{*contrav2}.
In the ergodic sums we consider, the first coordinate has no importance, thus it is sufficient to construct $\sigma\omega'$. 
Let it be periodic with period $N+k$ (that is $\sigma^{N+k+1}\omega'=\sigma\omega'$), and define its first $N+k$ coordinates to be $\omega_2,\omega_3,...,\omega_{N+k+1}$.
Now if $N'$ is arbitrary, express it modulo $N+k$ as $N'=p(N+k)+q$, where $p$ is a nonnegative integer, while $0\leq q < N+k$. 
Then the corresponding ergodic sum can be written as
\begin{equation}
\begin{split}
\frac{1}{N'} \sum_{n=1}^{N'} f(\sigma^n \omega') & = 
\frac{1}{N'} \sum_{n=1}^{p(N+k)} f(\sigma^n \omega') +
\frac{1}{N'} \sum_{n=1}^{q} f(\sigma^{p(N+k)+n} \omega') \\ & =
\frac{p(N+k)}{N'}\left(\frac{1}{p(N+k)}\sum_{n=1}^{p(N+k)} f(\sigma^n \omega')\right)+\frac{1}{N'} \sum_{n=1}^{q} f(\sigma^{p(N+k)+n} \omega')  =\circledast
\end{split}
\end{equation}
Using  the periodicity of $\sigma \omega'$ in the first sum, and the boundedness of $f$ in the second one we infer
$$ \circledast=
\frac{p(N+k)}{N'}\left(\frac{1}{N+k}\sum_{n=1}^{N+k} f(\sigma^n \omega')\right)+o(N').$$
Hence if $N'\to\infty$, the ergodic sum $\frac{1}{N'} \sum_{n=1}^{N'} f(\sigma^n \omega')$ converges to $\frac{1}{N+k}\sum_{n=1}^{N+k} f(\sigma^n \omega')$. 
Now by (\ref{contradict}) and $f\in\PCC^k(\Omega)$, we have \eqref{*contrav}. Thus by (\ref{avgestimate}), we deduce
\begin{equation}\label{*contrav2}
\frac{1}{N+k}\sum_{n=1}^{N+k} f(\sigma^n \omega')>\frac{-k\norm{f}+N(\alpha_{f, \max}^* +\varepsilon)}{N+k}>\alpha_{f, \max}^*+\frac{\varepsilon}{2},
\end{equation}
Hence $E_f(\alpha)\neq\emptyset$ for some $\alpha>\alpha_{f, \max}^*+\frac{\varepsilon}{2}$, which is obviously a contradiction. 
It concludes the proof. \end{proof}


\begin{remark}
	More general version of Lemma \ref{uniformity} can be found in \cite[Theorem 1.9]{SS2000}. In particular, the result would hold for continuous functions, rather than $\PCC$ functions. We will not, however, require such general result in our subsequent argument.
\end{remark}

Next, we will show that if $f \in PCC(\Omega)$, then there exists a periodic point in $\Omega$ for which the limit of the Birkhoff averages of $f$ equals  $\alpha_{f, \max}^*$.

\begin{lemma} \label{periodicmax}
Let $f\in\PCC^k(\Omega)$. 
Then there exists a periodic configuration $\omega$ such that $\lim_{N\to\oo}\frac{1}{N}\sum_{n=1}^N f(\sigma^n \omega) = \alpha_{f,\max}^*$.
\end{lemma}

\begin{proof}
We define a directed graph $G=(V,E)$ as follows: $V=\{0,1\}^k$, and there is an edge from $u\in V$ to $v\in V$ if roughly speaking $v$ is one of the possible shifted images of $u$, that is $v_i=u_{i+1}$ for $i=1,...,k-1$. 
Now we can think of the values of $f$ as weights on the vertices of $G$, while an arbitrary $\omega\in\Omega$ corresponds to an infinite walk $\Gamma_{\omega}$ in $G$. 
Moreover, the ergodic averages are simply the averages of weights along the vertices of finite subwalks of $\Gamma_\omega$.

For technical reasons, it is advantageous to put the weights on the edges and work with those ones: one of the convenient ways to do so is putting weight $f(u)$ on all the edges \textit{leaving} the vertex $u$. 
Denote the function $E\to\mathbb{R}$ obtained this way by $f$, too. 
Now the ergodic averages can be considered as the averages of weights along the edges of finite subwalks of $\omega$.

Consider now $\omega\in\Omega$ such that $\frac{1}{N}\sum_{i=1}^{N}f(\sigma^i \omega)\to\alpha_{f,\max}^{*}$.
Take the corresponding path $\Gamma_{\omega}$. 
As $V$ is finite, there exists a vertex which appears infinitely many times in $\Gamma_{\omega}$. 
By erasing the first few entries of $\omega$, or equivalently, erasing the first few edges of $\Gamma_{\omega}$, we might assume by abuse of notation that the first vertex $v$ of $\Gamma_{\omega}$ recurs infinitely many times. 
Now based on the recurrences of $v$, we can partition the infinite walk $\Gamma_{\omega}$ into closed, finite walks $\Gamma_{\omega}^{(1)},\Gamma_{\omega}^{(2)}, ...$ such that each such walk starts and ends with $v$, and in the meantime it does not hit $v$. 
Now it is simple to verify that the edge set (counted with multiplicities from now on) of each $\Gamma_{\omega}^{(i)}$ is the union of graph cycles, or in other words, it is the union of closed walks containing each of their edges precisely once. (One cycle might also appear multiple times in this decomposition.) Indeed, we can find a subpath $e_1e_2...e_r$ such that $e_1=e_r$, and there is no other repetition of edges in this subpath. 
Then $e_1e_2...e_{r-1}$ is a cycle, and its removal from $\Gamma_{\omega}^{(i)}$ results in a shorter closed walk starting and ending with $v$. Thus we can repeat the previous reasoning to find another cycle, if such exists and this procedure ends in finitely many steps.

Let us note now that there are only finitely many cycles in $G$ as it is a finite graph. 
Denote their set by $\mathcal{C}$. 
By the previous paragraph, up to the last edge of any $\Gamma_{\omega}^{(i)}$, the edge set of $\Gamma_{\omega}$ can be written as the union of these cycles, such that $C\in\mathcal{C}$ is used $\rho_{C,i}$ times. 
Thus the ergodic average corresponding to the subpath of the $\Gamma_{\omega}$ up to the last edge of $\Gamma_{\omega}^{(i)}$ is the following:
\begin{equation}
\frac{\sum_{C\in\mathcal{C}}\rho_{C,i}\sum_{e\in C} {f(e)}}{\sum_{C\in\mathcal{C}}\rho_{C,i}|C|}=\frac{\sum_{C\in\mathcal{C}}\rho_{C,i}|C|\sum_{e\in C} \frac{f(e)}{|C|}}{\sum_{C\in\mathcal{C}}\rho_{C,i}|C|}.
\label{pathaverage}
\end{equation}
Notice that it is simply a convex combination of the cycle averages $\sum_{e\in C} \frac{f(e)}{|C|}$. 
Hence the ergodic average in (\ref{pathaverage}) can be bounded from above by $\max_{C\in\mathcal{C}}\sum_{e\in C} \frac{f(e)}{|C|}$. 
Now by the choice of $\omega$ we also know that this ergodic average tends to $\alpha_{f, \max}^*$ as $i\to\infty$, hence
\begin{equation}
\alpha_{f,\max}^{*}\leq\max_{C\in\mathcal{C}}\sum_{e\in C} \frac{f(e)}{|C|}
\label{pathaverageestimate}
\end{equation}
also holds.

Now consider the infinite walk which goes along a cycle $C_0$ over and over again, where $C_0$ is chosen so that the above maximum is attained. 
Then $C_0$ together with a starting point uniquely determines a periodic configuration $\omega^*\in\Omega$ for which $\sigma^i\omega^*$ always equals the respective vertex of $C_0$. 
Moreover, it is simple to check that the ergodic averages tend to $\sum_{e\in C_0} \frac{f(e)}{|C_0|}$. 
Hence this limit must be $\alpha_{f,\max}^*$ by (\ref{pathaverageestimate}), as it is an upper estimate for all ergodic limits.
\end{proof}

\section{Continuity, discontinuity and support of the  spectrum}\label{*secgen}

By \cite{Fan_Feng_Wu_Recurrence}, we know that $S_f$ is necessarily upper semi-continuous for any continuous function. Moreover, it is continuous on $[\alpha_{f,\min}^*,\alpha_{f,\max}^*]$, while it vanishes outside of this interval. 
However it is not necessarily continuous at the endpoints of this interval.

\subsection{Denseness of $\PCC$ functions with discontinuous spectra}\label{SS-DDCS}
Recall an example of a $\PCC^3(\Omega)$ function with discontinuous spectrum from Example $\ref{*propdisc}$. 
In this section, we will show that functions in $\PCC(\Omega)$ with discontinuous spectrum form a dense subset of $C(\Omega)$.

\begin{theorem}\label{*ppropdisc}
Functions $h\in \PCC(\OOO)$ with $S_{h}(\aaa_{h,\max}^{*})>0$
are dense in $C(\OOO)$.
\end{theorem}

\begin{remark}\label{*rempropdisc}
Of course, a similar theorem is valid with $S_{h}(\aaa_{h,\min}^{*})>0$ in the conclusion and with a little extra technical effort one can show density in $C(\OOO)$ of those $f\in \PCC(\OOO)$ for which 
$S_{h}(\aaa_{h,\max}^{*})>0$ and $S_{h}(\aaa_{h,\min}^{*})>0$ hold simultaneously.
As Theorem \ref{contspectrum}
shows functions satisfying the conclusion of Theorem \ref{*ppropdisc}, or any of its above mentioned variants form a first category set in $C(\OOO)$.
\end{remark}

The main idea of the proof of Theorem \ref{*ppropdisc} is to show that given any continuous function, we can 
approximate it by a PCC function, and we further "perturb" that PCC function in an appropriate way so that its spectrum will be discontinuous.

\begin{proof}[Proof of Theorem $\ref{*ppropdisc}$]
Suppose $\eee>0$ and $f_{0}\in    C(\OOO)$ are arbitrary.
We need to find an $h\in \PCC(\OOO)$
 such that 
 \begin{equation}\label{*SP1*a}
 \|f_{0}-h\|<\eee \text{ and } S_{h}(\aaa^{*}_{h,\max})>0.
 \end{equation}
We will achieve this by the following. We first find $f \in \PCC^{\mathbf{k}}(\Omega)$ for suitably large $\mathbf{k}$ that approximates $f_0$. By Remark \ref{*remsup}, $\aaa^*_{f_0, \max} \approx \aaa^*_{f, \max}.$ Next we will "perturb" the function $f$ by adding another $\PCC$ function $g$. This function $g$ will be small in a sense that it does nothing to perturb $f$ for many points, but it will perturb just slightly near the points where the limit of Birkhoff averages of $f$ attains the maximum (i.e. $\alpha_{f,\max}^*$) to the point where $S_{f+g}$ discontinuous at the boundary of $L_{f+g}$. The sum $f+g$ will be our candidate for $h$.

By using a suitably large $\mathbf{k}$ choose $f\in \PCC^{\mathbf{k}}(\OOO)$
 such that $\|f-f_{0}\|<\eee/2$.
By Lemma \ref{periodicmax} select a periodic $\ooo'$  such that 
\begin{equation}\label{*SP2*a}
\lim_{N\to\oo}\frac{1}{N}\sum_{n=0}^{N-1}f(\sss^{n}\ooo')=\aaa^{*}_{f,\max}.
\end{equation} 
In this proof, as in \eqref{*SP2*a} we prefer to take Birkhoff sums with indices between $0$ and $N-1$, when taking limits it makes no difference.
We can assume that there is a finite string of $0$s and $1$s, denoted by $A$
 such that  $\ooo'=A^{\oo}$, by not necessarily using the prime period we can also suppose that $k_{A}=|A|$, the length of $A$ is a multiple of $\mathbf{k}$.

 Now we select a string $B$ of length $k_{A}$. 
If $A\not=0^{   k_{A}}$ then we let $B=0^{   k_{A}}$, if $A=0^{   k_{A}}$ then we let $B=1^{   k_{A}}$. 
Without limiting generality in the sequel we assume that $B=0^{   k_{A}}$.
 
 By using a suitably large number $\ell $, to be fixed later,
 we consider strings $X=(A^{2\ell })AABAA$ and
 $Y=(A^{2\ell })ABAAA$.
 
 Set $\mathbf{H}=\{X,Y  \}^{\oo}$. We will later show that this set will be contained in $E_h(\alpha_{h,\max}^*)$ (where $h$ will be defined by perturbing $f$ slightly). We note that it is easy to see that $\dim_H \mathbf{H} > 0$, since by Hutchinson's theorem,
 $2\cdot (2^{-(2\ell +5)   k_{A}})^{\dim_{H}\mathbf{H}}=1$, which gives
 $\dim_{H}\mathbf{H}=1/((2\ell +5)   k_{A})$. This would imply $E_h(\alpha_{h,\max}^*) > 0$.
 
 \begin{claim}
There exists $\alpha_\ell \in \RR$ such that
 \begin{equation}\label{*SP3*a}
 \lim_{N\to\oo} \frac{1}{N}\sum_{n=0}^{N-1} f(\sss^{n}\ooo)=\aaa_{\ell}\leq \aaa^{*}_{f,\max} \text{ for any }\ooo\in\mathbf{H}.
 \end{equation}
 \end{claim}
 
 \begin{proof}
 First we consider the sum $\sum_{n=0}^{2(\ell+5)k_A-1} f(\sss^{n}\ooo)$ for any $\omega \in \mathbf{H}$. We select $\omega^A \in [AA]$, $\omega^{AB} \in [ABA]$. Since $f \in \PCC^{\mathbf{k}}$, and $\mathbf{k}$ divides $k_A$, 
 the values of $f(\sigma^{n}\omega^A);\ n=0,...,k_{A}-1$ and  $f(\sigma^{n}\omega^{AB});\ n=0,...,2k_{A}-1$ are independent of our choice of  
 $\omega^A \in [AA]$ and $\omega^{AB} \in [ABA]$.
  Hence, there exists a constant $\Sigma_\ell$ such that for any $\omega \in \mathbf{H}$, we have
 $$\Sigma_\ell = \sum_{n=0}^{(2\ell+5)k_A-1} f(\sigma^n \omega) 
 = (2\ell+3) \sum_{n=0}^{k_A - 1} f(\sigma^n\omega^A) +   \sum_{n=0}^{2k_A-1} f(\sigma^n \omega^{AB}).
 $$
Define $\alpha_\ell$ so that it satisfies $2(\ell + 5)k_A \aaa_\ell = \Sigma_{\ell}$. 
Let $N$ be greater than $2(\ell+5)k_A$, and write $N = 2(\ell+5)k_AM_N + R_N$ for some positive integer $M_N$ and $R_N \in \{0, 1, \ldots, 2(\ell+5)k_A-1\}$. Thus
 \[ \sum_{n=0}^{N-1} f(\sigma^n \omega) = 2(\ell+5)k_AM_N\alpha_\ell + \sum_{n={N-R_N}}^N f(\sigma^n\omega) \]
 for any $\omega \in \mathbf{H}$. Thus, we obtain (\ref{*SP3*a}) by dividing both sides of the equation by $N$ and letting $N \to \infty$.
\end{proof}
 
 Now we return to the proof of Theorem \ref{*ppropdisc}.
 Next we construct the perturbation function $g$. Put $m=\ell +7$ and
\begin{equation}\label{*SP3*b}
C_{m}=\{U_{1}U_{2}...U_{m}\ooo_{0}\ooo_{1}...: U_{i}\in \{ X,Y \},\  i=1,...,m,\  \ooo_{j}\in \{ 0,1 \},\  j=0,1,...  \}.
\end{equation} 
We take the following finite union of cylinder sets in $\OOO$
$$\mathbf{P} =\cup_{i=0}^{\ell -1}\sss^{i   k_{A}}C_{m}.$$ 

Next we define our perturbation function $g\in \PCC^{m   k_{A}}(\OOO).$
If $\ooo\in \mathbf{P} $ then we set $g(\ooo)=\eee/4$, otherwise put $g(\ooo)=0$.

\begin{claim}\label{*hfg}
If $\ell $ is sufficiently large then for
$h=f+g$ and for any $\ooo\in\mathbf{H}$ we have 
\begin{equation}\label{*SP7*a} b^* :=
\frac{1}{(2\ell + 5)k_{A} }\sum_{j=0}^{(2\ell + 5)k_{A} -1}h(\sss^{j+t(2\ell + 5)k_{A} }\ooo) >
\aaa^{*}_{f,\max}+\frac{\eee}{32   k_{A}} \text{ for }t=0,1,...\, .
\end{equation}
\end{claim}

\begin{proof}
Take and fix an arbitrary $\ooo\in\mathbf{H}$. 
Recall that $|X|=|Y|=(2\ell +5)   k_{A}$.
By our definition of $X$ and $Y$ we have 
\begin{equation}\label{*SP4*a}
\frac{1}{(2\ell + 5)k_{A} }\sum_{j=0}^{(2\ell + 5)k_{A} -1}f(\sss^{j+t(2\ell + 5)k_{A} }\ooo)=\aaa_{\ell}
\text{ for any }t\in \{ 0,1,... \}.
\end{equation}

From the choice of $\ooo$ and $A$ it is also clear that
\begin{equation}\label{*SP5*a}
\frac{1}{2\ell    k_{A}}\sum_{j=0}^{2\ell    k_{A}-1}f(\sss^{j+t(2\ell + 5)k_{A} }\ooo)=\aaa^{*}_{f,\max}
\text{ for any }t\in \{ 0,1,... \}.
\end{equation}
Hence,
\begin{equation}\label{*SP5*b}
\aaa_{\ell}\geq \frac{2\ell    k_{A}\cdot \aaa^{*}_{f,\max}+5   k_{A}\aaa_{f,\min}}{(2\ell + 5)k_{A} }\to \aaa^{*}_{f,\max}
\text{ as }\ell \to\oo.
\end{equation}

Next we look at the averages of $g$. Observe that if $U_{i}\in \{ X,Y \}$ then there is a maximal substring of $U_{i}$
which consists of consecutive zeros. 
This is the one which contains $B$, and
of course might contain some zeros from the end/beginning of the $A$s before/after $B$
in $U_{i}$. 
This and the definition of $\mathbf{P}$ and $g$ imply that for $\ooo\in\mathbf{H}$
\begin{equation}\label{*SP5b*a}
g(\sss^{j}\ooo)>0\text{ holds iff }j=i   k_{A}+t(2\ell +5)   k_{A},\   i=0,...,\ell -1,\   
t=0,1,... \, . 
\end{equation}

Therefore,
\begin{equation}\label{*SP5*c}
\frac{1}{(2\ell + 5)k_{A} }\sum_{j=0}^{(2\ell + 5)k_{A} -1}g(\sss^{j+t(2\ell + 5)k_{A} }\ooo)=\frac{\ell \eee}{4(2\ell + 5)k_{A} }
\text{ for any }t\in \{ 0,1,... \}.
\end{equation}

Now we determine how large $\ell $ should be. Indeed, we select an $\ell $  such that 
\begin{equation}\label{*SP6*b}
\ell \cdot \frac{\eee}{8}> 5   k_{A} (\aaa^{*}_{f,\max}-\aaa_{f,\min}) \text{ and }
\frac{\ell }{8(2\ell +5)}>\frac{1}{32}.
\end{equation}
These inequalities will grant us that
\begin{align}\label{*SP6*a}
 \frac{2\ell    k_{A}\aaa^{*}_{f,\max}+5   k_{A}\aaa_{f,\min}+\ell \frac{\eee}{4}}{(2\ell + 5)k_{A} }
&> \frac{2\ell    k_{A}\aaa^{*}_{f,\max}+5   k_{A}\aaa^{*}_{f,\max}+\ell \frac{\eee}{8}}{(2\ell + 5)k_{A} } \\
&>\aaa^{*}_{f,\max}+\frac{\eee}{32   k_{A}}. \nonumber
\end{align}


From \eqref{*SP5*a}, \eqref{*SP5*b}, \eqref{*SP5*c} and \eqref{*SP6*a}, it follows that if
$h=f+g$ then for $\ooo\in\mathbf{H}$ we have \eqref{*SP7*a}.

\end{proof}

Now we return again to the proof of Theorem \ref{*ppropdisc}.
Claim \ref{*hfg} implies that $\mathbf{H}\sse E_{h}(b^{*})$ and hence $S_{h}(b^*)=
\dim_{H}E_{h}(b^*)>0$.

If we can verify that $b^*=\aaa^{*}_{h,\max}$ then we are done.
We need to show that if
\begin{equation}\label{*SP8*a}
\lim_{N\to\oo} \frac{1}{N}\sum_{n=0}^{N-1} h(\sss^{n}\ooo)=\aaa \text{ then }\aaa\leq b^*.
\end{equation}
Suppose that we have a fixed $\ooo\in \OOO$ for which the limit in \eqref{*SP8*a}
exists and equals $\aaa$.

Now we subdivide $\ooo$ into finitely or infinitely many substrings in the following way
$$\ooo=Z_{0}W_{1}Z_{1}W_{2}Z_{2}...$$
where $Z_{0}$ might be the empty string, the other strings are non-empty.
For any $j$ the strings $W_{j}\in \{ X,Y \}^{d_{j}}$, where $1\leq d_{j}\leq +\oo.$
The strings $Z_{j}$ do not contain any substring of the form $X$ or $Y$ and they can be finite, or infinite.
In case one of the $Z_{j}$s is infinite then  there exists $N_{1}$  such that for all
$n\geq N_{1}$, $g(\sss^{n}\ooo)=0$
and hence $\aaa\leq \aaa^{*}_{f,\max}<b^{*}$.

Hence from now on we can suppose that the $Z_{j}$s are finite.

If one of the $W_{j}$s is infinite then one can find $N_{1}$  such that 
$\sss^{N_{1}}\ooo\in \mathbf{H}$
and hence $\aaa=b^*$ by \eqref{*SP7*a}.

Hence from now on we can suppose that all the $W_{j}$s are finite.

Since for any $k\in \mathbf{N}$ we have $\ooo\in E_{h}(\aaa)$ iff $\sss^{k}\ooo\in E_{h}(\aaa)$
we can suppose that $Z_{0}=\ess$ and hence
$\ooo=W_{1}Z_{1}W_{2}Z_{2}...$.
Choose $k_{j}$, $j=1,2,...$  such that the substring $W_{j}Z_{j}$ of $\ooo$
starts at $\ooo_{k_{j}}$, 
that is $W_{j}Z_{j}=\ooo_{k_{j}}\ooo_{k_{j}+1}...\ooo_{k_{j+1}-1}.$ 
We denote by $k_{j}'$ the place where $Z_{j}$ starts, that is,
$W_{j}=\ooo_{k_{j}}\ooo_{k_{j}+1}...\ooo_{k_{j}'-1}$ and
$Z_{j}=\ooo_{k_{j}'}\ooo_{k_{j}'+1}...\ooo_{k_{j+1}-1}$.


Suppose that we have a $j$ for which
\begin{equation}\label{*SP11*a}
\text{there exists }n\in \{ k_{j},...,k_{j+1}-1 \}
 \text{
  such that 
 }
 g(\sss^{n}\ooo)>0.
\end{equation}
We denote the set of such $j$s by $\mathbf{J}$.

Then $g(\sss^{n}\ooo)=\eee/4$.
We define $n_j$ to be the maximal $n$ satisfying the inequality in 
\eqref{*SP11*a}.
Since $Z_j$ does not contain a substring of the form $X$ or $Y$, $\sigma^n \omega \notin \mathbf{P}$ for any $k_j' < n < k_{j+1}$. 
Hence $n_j < k_j'$.
Moreover, by the definition of $g$ and $\mathbf{P} $ we have
$$n_j=k_{j}'-m(2\ell + 5)k_{A} +(\ell -1)   k_{A}.$$
Put
$$k_{j}''=n_j-(\ell -1)   k_{A}+(2\ell + 5)k_{A} .$$

Then by the definition of $g$
\begin{equation}\label{*SP12*a}
\sss^{k_{j}''}\ooo|(2\ell + 5)k_{A} \in \{X, Y\} 
\text{ and }
\sss^{k_{j}}\ooo|(k_{j}''-k_{j})\in \{X, Y\} ^{(k_{j}''-k_{j})/(2\ell + 5)k_{A} },
\end{equation}
where $(k_{j}''-k_{j})/(2\ell + 5)k_{A} $ is an integer, that is 
$\sss^{k_{j}}\ooo|(k_{j}''-k_{j})$ starts with a long string of $X$s and $Y$s.
Hence
\begin{equation}\label{*SP13*a}
\frac{1}{k_{j}''-k_{j}}\sum_{n=k_{j}}^{k_{j}''-1}
g(\sss^{n}\ooo)
=
\frac{\ell \eee}{4(2\ell + 5)k_{A} }.
\end{equation}
It is also clear that
\begin{equation}\label{*SP13*b}
\frac{1}{k_{j+1}-k_{j}''}\sum_{n=k_{j}''}^{k_{j+1}-1}g(\sss^{n}\ooo )=0.
\end{equation}

Suppose that $\ddd>0$ is given.
We want to find $N_{\ddd}$  such that for $N\geq N_{\ddd}$
we have
\begin{equation}\label{*SP13*c}
 \frac{1}{N}\sum_{n=0}^{N-1} h(\sss^{n}\ooo )< b^*+\ddd.
\end{equation}

We can suppose that $\mathbf{J}$ is infinite since otherwise
there exists $N_{1}$
 such that 
 $h(\sss^{n}\ooo )=f(\sss^{n}\ooo )$ for $n\geq N_{1}$
 and $\aaa\leq \aaa^{*}_{f,\max}<b^*$ holds.
 
 We will obtain $N_\delta$ by splitting $\ooo$ into two infinite substrings: The "good part" and the "bad part." The "good part" can be obtained as the concatenation of the substrings
 $\sss^{k_{j}}\ooo|(k_{j}''-k_{j}),$ $j\in \mathbf{J}$, while the "bad part" of $\ooo$ is the "rest" of $\ooo$,
 that is what is left of $\ooo$ if we delete from it the good part. We denote this bad part $\ooo^b$.
 To be more specific
 if $j\not\in \mathbf{J}$
 then we take the string
 $\sss^{k_{j}}\ooo|(k_{j+1}-k_{j})$, otherwise if
 $j\in \mathbf{J}$
 then we take the string
 $\sss^{k_{j}''}\ooo|(k_{j+1}-k_{j}'')$
and concatenate these strings. To achieve our goals of obtaining $N_\delta$, we will be observing Birkhoff averages along the "good" part $\mathbf{N}^g := \cup_{j\in\mathbf{J}}\{ k_{j},...,k_{j}''-1 \}$ and the "bad" part $\mathbf{N}^b = \{ 0,1,... \}\sm \mathbf{N}^{g}$. In particular, we will be at some point evaluating Birkhoff sum on the point $\omega^b$ rather than $\omega$; we will explain why this works, particularly when we verify equation \eqref{*SP18*a}.

Using \eqref{*SP7*a}, \eqref{*SP13*a}, and the definition of the strings $X$ and $Y$
it is clear that if $j\in J$ then
\begin{equation}\label{*SP15*a}
\frac{1}{k_{j}''-k_{j}}\sum_{n=k_{j}}^{k_{j}''-1}h(\sss^{n}\ooo )=b^*.
\end{equation} 

We also know that if $n \in \mathbf{N}^b$
then $g(\sss^{n}\ooo )=0$ and hence $h(\sss^{n}\ooo )=f(\sss^{n}\ooo )$.

Moreover, whenever $t \in \NN$ satisfies the inequality $(t+1)(2\ell + 5)k_{A} \leq k_{j}''-k_{j}$, for some $j\in\mathbf{J}$, then
\begin{equation}\label{*SP16*a}
\frac{1}{(2\ell + 5)k_{A} }\sum_{n=k_{j}+t(2\ell + 5)k_{A} }^{k_{j}+(t+1)(2\ell + 5)k_{A} -1}h(\sss^{n}\ooo )=b^{*}
\end{equation}
holds as well.

From \eqref{*SP16*a} and the boundedness of $h$
it follows that we can select $N_{\ddd}'$
 such that for $N>N_{\ddd}'$
\begin{equation}\label{*SP17*a}
\frac{1}{ \#\{ n\in \mathbf{N}^{g}: n<N \}}\sum_{n\in \mathbf{N}^{g},\  n<N}h(\sss^{n}\ooo )< b^{*}+\frac{\ddd}{2}.
\end{equation}
Denote $\#\{ n\in \mathbf{N}^{b}: n<N \}$ by $\nu_{b}(N)$.

Next we need to estimate 
\begin{equation}\label{*SP17*b}
\frac{1}{\nu_{b}(N)}\sum_{n\in \mathbf{N}^{b},\  n<N}h(\sss^{n}\ooo )=
\frac{1}{\nu_{b}(N)}\sum_{n\in \mathbf{N}^{b},\  n<N}f(\sss^{n}\ooo ).
\end{equation}

A little later we will show that
\begin{equation}\label{*SP18*a}
\frac{1}{\nu_{b}(N)}\sum_{n\in \mathbf{N}^{b},\  n<N}f(\sss^{n}\ooo )=
\frac{1}{\nu_{b}(N)}\sum_{n=0}^{\nu_{b}(N)-1}f(\sss^{n}\ooo ^{b}).
\end{equation}
Next we show that if we verified this then we can complete our proof. 
Indeed by Lemma
\ref{uniformity}
$$\limsup_{N'\to\oo}\frac{1}{N'}\sum_{n=0}^{N'-1}f(\sss^{n}\ooo ^{b})\leq \aaa^{*}_{f,\max}$$
and hence we can select $N_{\ddd}\geq N_{\ddd}'$
 such that if $N\geq N_{\ddd}$ then $\nu_{b}(N)$ is sufficiently large to have
$$\frac{1}{\nu_{b}(N)}\sum_{n=0}^{\nu_{b}(N)-1}f(\sss^{n}\ooo ^{b})\leq \aaa^{*}_{f,\max}+\frac{\ddd}{2}.$$
By \eqref{*SP18*a} this yields that
$$\frac{1}{\nu_{b}(N)}\sum_{n\in \mathbf{N}^{b},\  n<N}f(\sss^{n}\ooo )<\aaa^{*}_{f,\max}+\frac{\ddd}{2}<
b^*+\frac{\ddd}{2}.$$
From this, \eqref{*SP17*a}, and \eqref{*SP17*b}, it follows that for $N>N_{\ddd}$
$$ \frac{1}{N}\sum_{n=0}^{N-1} h(\sss^{n}\ooo )<b^*+\frac{\ddd}{2}.$$
Since a suitable $N_{\ddd}$ can be chosen for any $\ddd>0$ we proved that $\aaa\leq b^*.$

Hence, to complete the proof of the theorem we need to verify \eqref{*SP18*a}.
But this is not difficult. 
Since $f\in \PCC^{\mathbf{k}}(\OOO)$ we know that $f(\sss^{n}\ooo )$
depends only on the string $\sss^{n}\ooo |\mathbf{k}$. Observe that during the definition of $\ooo^{b}$ we concatenate strings which start with
a string $A$ and $A$ is of length $   k_{A}>\mathbf{k}$.
Indeed, if $j\not \in\mathbf{J}$
then during the definition we concatenate the string $\sss^{k_{j}}\ooo|(k_{j+1}-k_{j})=W_{j}Z_{j}$, and $W_{j}$ starts with $X$ or $Y$ and they both start with $A$. 
If $j\in \mathbf{J}$ then we take the string $\sss^{k_{j}''}\ooo|(k_{j+1}-k_{j}'')$
and by \eqref{*SP12*a} this string starts with $A$.

We can define a function $\ppp:\{ 0,1,... \}\to \mathbf{N}^{b}$
the following way. 
For $n\in \{ 0,1,... \}$ if we take $\ooo_{n}^{b}$ then this entry
corresponded to exactly one entry $\ooo_{\ppp(n)}$ of $\ooo$ and belonged to a concatenated string making up $\ooo^{b}$. 
Suppose that $k_{j}\leq \ppp(n)<k_{j+1}$.
If $\ppp(n)\leq k_{j+1}-\mathbf{k}$ then the strings $\sss^{n}\ooo^{b}|\mathbf{k}$ and $\sss^{\ppp(n)}\ooo|\mathbf{k}$ are identical and hence $f(\sss^{n}\ooo^{b})=f(\sss^{\ppp(n)}\ooo).$
If $\ppp(n)>k_{j+1}-\mathbf{k}$ then there is an $n'<n+\mathbf{k}$
such that $\ppp(n')=k_{j+1}$.  
By our concatenation procedure it is clear that the strings
$\sss^{n}\ooo^{b}|(n'-n)$ and
$\sss^{\ppp(n)}\ooo|(n'-n)$ are identical.
It is also clear that  $\ppp(n')=k_{j+1}$ and
$\sss^{\ppp(n')}\ooo|   k_{A} =A$, since
we take the first $   k_{A}$ entries of a string
which equals $X$ or $Y$.
Now recall our earlier observation 
that $\ooo^{b}$ was obtained by the concatenation of strings which start with $A$.
Hence $\sss^{n'}\ooo^{b}$ starts with the string $A$.
This implies again that $f(\sss^{n}\ooo^{b})=f(\sss^{\ppp(n)}\ooo).$ 
\end{proof}

\subsection{A generic continuous function has a continuous Birkhoff spectrum}\label{SS-GCBS}
We have seen in the previous subsection that functions with dicontinuous spectrum form a dense set in $C(\Omega)$. 
Next we will show that the set of such functions is of first category.


\begin{theorem} \label{contspectrum}
For the generic continuous function $f\in C(\Omega)$, we have that $S_f$ is continuous on $\mathbb{R}$.
\end{theorem}

\begin{remark} \label{coboundary}
This theorem implies that the set of continuous functions with discontinuous Birkhoff spectrum is a  set of first category. 
This set includes functions which are cohomologous to a constant, as we observed in Example \ref{cob_ex}, hence this is a possible way to see that these functions form a set of first category.
\end{remark}

To prove Theorem \ref{contspectrum}, we need the following lemma, which shows that one can "perturb" a PCC function so that the new function would have a continuous spectrum.


\begin{lemma} \label{perturbation}
Let $f\in \PCC^k(\Omega)$ and let $\varepsilon>0$. 
Then there exists $g\in C_0(\Omega)$ such that $\norm{g}<\varepsilon$, $S_{f+g}$ vanishes at $\alpha_{f+g,\max}^*$ and
$\alpha_{f,\min}^*-\eee\leq \alpha_{f+g,\min}^*\leq\alpha_{f+g,\max}^*
\leq \alpha_{f,\max}^*+\eee$.
\end{lemma}
 
\begin{proof}
Let $f \in PCC^k(\Omega)$ and let $\eee > 0$. 
Let $\omega^*$ be a periodic point with prime period $p$ for which $\frac{1}{p}\sum_{n=1}^p f(\sigma^n \omega^*) = \alpha_{f, \max}^*$ (which exists by Lemma $\ref{periodicmax}$). Let $g_0(\omega) = \min_{i=1, \ldots, p} \{d(\omega, \sigma^i \omega^*)\}$, and let $g = -\eee g_0 + c$ , where $c = \eee \int g_0 \, d\lambda$, which implies $\int g \, d\lambda = 0$. 
Since $\lll(\OOO)=\diam(\Omega) = 1$, it is clear that $\norm{g} < \eee$.

Given $E \subset \NN$, we denote by  ${\mathbf d}(E)$ the   density of the set $E$, that is $\lim _{N\to\oo}\frac{\#(E\cap [1,N])}{N}$ (if it exists). We let
\[H_{\omega^*} := \{\omega \in \Omega: \omega|_E = \omega^* \text{ for some } E \subset \NN \text{ for which } {\mathbf d}(E^c) = 0 \}, \]
where $\omega|_E$  denotes the concatenation of $\ooo_{j}$, $j\in E$.
We will show that $E_{f+g}(\alpha_{f+g,\max}^*) \subset H_{\omega^*}$, and then we observe that $\dim_H H_{\omega^*} = 0$.

By using \eqref{*avgN4}  from Lemma \ref{uniformity} one can see that $\alpha_{f+g, \max}^* \leq \alpha_{f, \max}^* + c$.
Since $g_{0}(\sss^{n}\ooo^{*})=0$ for any $n$, we obtain $\alpha_{f+g, \max}^* \geq \alpha_{f, \max}^* + c$,
and hence $\alpha_{f+g, \max}^* = \alpha_{f, \max}^* + c$.
Let $\omega \in E_{f+g}(\alpha_{f+g,\max}^*)$. 
Then we must have
\[\lim_{N \to \infty} \left( \frac{1}{N} \sum_{n=1}^N f(\sigma^n \omega) - \frac{\eee}{N} \sum_{n=1}^N g_0(\sigma^n \omega) \right) = \alpha_{f, \max}^* \, , \]
and this is only possible if $\frac{1}{N} \sum_{n=1}^N f(\sigma^n \omega) \to \alpha_{f, \max}^*$, and, in particular, 
$$\frac{1}{N} \sum_{n=1}^N g_0(\sigma^n \omega) \to 0 \text{ as } N \to \infty.$$
This implies  that the set $$J_\omega := \{n \in \NN: g_0(\sigma^n \omega) \geq 2^{-p} \}$$ has zero density.
Observe that if $g_{0}(\sss^{n}\ooo)<2^{-p}$ for $n=j',...,j'+l$ then there exists $i\in \{ 0,...,p-1 \}$  such that 
\begin{equation}\label{*ooo*}
(\sss^{n}\ooo)_{j'}^{j'+l+p}=\sss^{i}\ooo^{*}|l+p+1.
\end{equation}

The case when $J_\omega$ is finite is much easier and is left to the reader, we detail only the case when $J_\omega$ is infinite.

Suppose we enumerate $J_\omega = \{j_1, j_2, j_3, \ldots \}$ in the increasing order and we set $j_0 := 1$. 
Then for each $k \in \NN\cup \{ 0 \}$, 
there exists $i_{k}\in \{ 0,...,p-1 \}$ such that the (possibly empty) string
$\ggg(j_{k}):=(\omega)_{j_{k}+1}^{j_{k+1} - 1}$ equals  $\sss^{i_{k}}\omega^*|j_{k+1}-j_{k}-1$. 
Hence, we have
\[\omega|_{J_\omega^c} = \gamma(j_0)\gamma(j_1)\gamma(j_2)\cdots .\]

Since $\ooo^{*}$ is periodic we can choose $m_{k}\in \{ 0,...,p-1 \}$
 such that if 
$\gamma^*(j_k)=\sss^{m_{k}}\ggg(j_{k})$, that is we throw away the first $m_{k}$ entries of $\gamma(j_k)$, then
\[\gamma^*(j_0) \gamma^*(j_1) \gamma^*(j_2) \cdots = \omega^* \, . \]
Put $F=\cup_{k}\{ j_{k},...,j_{k}+m_{k} \}$.
Then $F \subset \bigcup_{i=0}^{p-1} J_\omega + i$ (where $A + b = \{a + b : a \in A\}$ for any $A \subset \NN$ and $b \in \NN$), which has a zero density. 
Setting $E = F^c$, we get $\omega|_E = \omega^*$. 
Hence, $\omega \in H_{\omega^*}$, which shows that $E_{f+g}(\alpha_{f+g, \max}^*) \subset H_{\omega^*}$.

We now show that $\dim_H H_{\omega^*} = 0$. 
Consider the set $H_{\mathbf{0}} := \{\ooo \in \Omega: {\mathbf d}(\{i \in \NN: \omega_i = 1\}) = 0\}$. 
Due to Example \ref{examplecontspect} we see that $\dim_H(H_\mathbf{0}) = 0$ as it equals $S_f(0)$ for $f$ defined in that example. 
Given $\ooo\in\OOO$ and $i\in\NN$ we set $\nu(i,\ooo)=\# \{ j: \ooo_{j}=0,\  j\leq i \}.$
We define a map $h: \Omega \to \Omega$ as follows: $h(\omega) = h_1h_2h_3\ldots$, where
\[h_i := \left\{ \begin{array}{ll} 
\omega_{\nu(i,\ooo)}^* & \text{if } \omega_i = 0 \\
1 - \omega_{\nu(i,\ooo)+1}^* & \text{if } \omega_i = 1.
\end{array}
\right. \]
It is easy to see that  $h$ is Lipschitz. 
One can also verify easily that $h(H_{\mathbf{0}}) \supset H_{\omega^*}$. 
Therefore, $0\leq \dim_H(H_{\omega^*}) \leq \dim_H(h(H_\mathbf{0})) = 0$.
\end{proof}

What remains from the proof of Theorem  \ref{contspectrum} is rather standard:

\begin{proof}[Proof of Theorem \ref{contspectrum}]
It suffices to prove that a generic continuous function $h$ has continuous spectrum at the points $\alpha_{h,\min}^*$ and $\alpha_{h,\max}^*$, and due to symmetry reasons, it suffices to prove the continuity in $\alpha_{h,\max}^*$ (if it holds in a residual set, the other also does in another residual set, and the intersection of these sets is still residual). We will prove in fact that the set
\[Z=\{h\in C(\Omega) : S_h \text{ is not continuous at } \alpha_{h,\max}^*\}\]
is meager. 
Note that we know that $S_h$ is concave and achieves its maximum at $\int h \, d\lambda$, hence
\[Z=\bigcup_{n=1}^{\infty}Z_{\frac{1}{n}},\]
where
\[Z_\theta=\left\{h\in C(\Omega) : S_h(x)>\theta \text{ for all } x\in\left[\int h \, d\lambda,\alpha_{h,\max}^*\right]\right\}.\]
Now it suffices to prove that each $Z_\theta$ is nowhere dense, and clearly it is enough to consider small enough $\theta<1$. 
To this end, take arbitrary $f\in\PCC^k(\Omega)$ for some $k$, and $\varepsilon>0$. 
By Lemma \ref{perturbation}, we can find $f+g$ in the $\varepsilon$-neighborhood of $f$ such that it has continuous spectrum at $\alpha_{f+g,\max}^*$. 
Then $\alpha_{f+g,\max}^*>\int (f+g)\geq \alpha_{f+g,\min}^*$ necessarily holds, as $S_{f+g}(\int (f+g)d\lll)=1$. 
Now by continuity, we can take $x\in\left[\int h,\alpha_{f+g,\max}^*\right]$ such that $0<S_{f+g}(x)<\frac{\theta}{2}$. By its concavity 
$S_{f+g}$ is monotone decreasing on  $[\int h,\alpha_{f+g,\max}^*]$
hence we can assume that 
\[x-\alpha_{f,\min}^*\geq \alpha_{f,\max}^*-x.\]
Now apply Theorem \ref{Thm:NormCont} for $f+g$ with
\begin{equation} \label{nghbdef}
\varepsilon'=\min\Big \{\frac{\theta}{2},\alpha_{f,\max}^*-x\Big \}.
\end{equation}
It guarantees that $0<S_h(x)<\theta$ for any $h$ with $\norm{h-(f+g)}<\delta'$ for some $\delta'>0$. 
Moreover, if $h$ and $f+g$ are close enough to each other, we also have that their integral cannot differ by much, hence we also have that $x\in\left[\int h,\alpha_{h,\max}^*\right]$. 
Consequently, if $h$ is in a sufficiently small neighborhood of $f+g$ satisfying both this integral condition and what is given by (\ref{nghbdef}), then $h$ is not in $Z_\theta$. 
It yields that $Z_\theta$ is nowhere dense, as $PCC(\Omega)$ is dense, and in the neighborhood of an arbitrary $f$ belonging to this set we constructed an open ball which is disjoint from $Z_\theta$. 
It concludes the proof.
\end{proof}

\subsection
{Supports of generic spectra are in $(\aaa_{f,\min},\aaa_{f,\max})$}
 \label{SS-gensp}
 In Example \ref{examplecontspect} we saw a very simple $\PCC$ function
 for which the range of the function $[\aaa_{f,\min},\aaa_{f,\max}]$
 coincides with the support of the spectrum
 $[\aaa_{f,\min}^{*},\aaa_{f,\max}^{*}]$.
In this subsection we verify that for the generic continuous function this is not true. In fact, we prove a little more; we show that the set of functions for which $[\aaa_{f,\min}^{*},\aaa_{f,\max}^{*}] \subset (\aaa_{f,\min},\aaa_{f,\max})$ is open and dense.

\begin{theorem}\label{*typsup}
For a dense open set $\cag\sse C(\OOO)$ we have 
\begin{equation}\label{*afmsz}
\alpha_{f,\min}<\alpha_{f,\min}^*<\alpha_{f,\max}^*<\alpha_{f,\max}
\end{equation}
hence the generic $f\in C(\OOO)$ satisfies \eqref{*afmsz}.
\end{theorem}

\begin{proof}
It suffices to prove that each inequality in (\ref{*afmsz}) holds in a dense open subset of $C(\Omega)$, and due to symmetry, it is sufficient to prove that $\alpha_{f,\min}^*<\alpha_{f,\max}^*$ and $\alpha_{f,\max}^*<\alpha_{f,\max}$ hold in dense open subsets, respectively. Given Remark \ref{*remsup}, it immediately follows that each of these inequalities holds in an open subset, thus we only have to keep an eye on denseness.

Consider first $\alpha_{f,\min}^*<\alpha_{f,\max}^*$. By Theorem \ref{contspectrum} we know that $S_f$ is continuous 
for $f\in G_{1}$
with a dense subset $G_1\subset C(\Omega)$. However, for $\aaa_{\lll}=\int f d\lll$
we have $S_{f}(\aaa_{\lll})=1$, and $S_{f}(\aaa_{f,\min}^{*})=S_{f}(\aaa_{f,\max}^{*})=0$,
hence
\begin{equation}\label{*afszmf1bnew}
\aaa_{f,\min}^{*}<\aaa_{f,\max}^{*}.
\end{equation}
It yields that for any $f\in G_1$ we have $\alpha_{f,\min}^*<\alpha_{f,\max}^*$, thus this inequality holds in a dense subset indeed.

Let us consider now $\alpha_{f,\max}^*<\alpha_{f,\max}$. We know that functions $f\in\PCC(\Omega)$ are dense in $C(\Omega)$. Consider such a function $f$, we have $f\in\PCC^k(\Omega)$ for some $k>0$. By Lemma \ref{periodicmax} we know that there exists a periodic configuration $\omega_f$ with $\lim_{N\to\oo}\frac{1}{N}\sum_{n=1}^N f(\sigma^n \omega_f) = \alpha_{f,\max}^*$. 
If $\aaa_{f,\max}^{*}<\aaa_{f,\max}$ then we are done.
Hence we can suppose that $\aaa_{f,\max}^{*}=\aaa_{f,\max}.$

Assume first that $\omega_f$ can be chosen such that $\omega_f$ is neither identically $1^\infty$ nor $0^\infty$. 
Then we can choose a substring $A$ of length $k$ such that $f$ is maximal on $[A]$ and $A$ is neither $11\cdots1$ nor $00\cdots0$ (i.e. blocks of $k$ many $1$s or $0$s, respectively), actually, by periodicity of $\omega_{f}$ any substring of $A$, different from  $11\cdots1$ and $00\cdots0$ is suitable.
 Now for given $\varepsilon>0$ define $g\in\PCC^k(\Omega)$ such that $g = f + \eee \ind_{[A]}$. Select a periodic $\omega_g$  for which $\lim_{N\to\oo}\frac{1}{N}\sum_{n=1}^N g(\sigma^n \omega_g) = \alpha_{g,\max}^*$, the existence of $\omega_g$ is again guaranteed by Lemma \ref{periodicmax}.
The relative frequency of the substring $A$ in $\omega_g$ is strictly smaller than $1$, as $A$ contains both $0$s and $1$s, hence at least $1/k$ of the substrings start with a binary digit different from the first entry in $A$. Thus we can conclude
\[\lim_{N\to\oo}\frac{1}{N}\sum_{n=1}^N g(\sigma^n \omega_g)-\lim_{N\to\oo}\frac{1}{N}\sum_{n=1}^N f(\sigma^n \omega_f)< \|g-f\|=\eee,\]
hence
\[\alpha_{g,\max}^*-\alpha_{f,\max}^*<\eee.\]
However,
\[\alpha_{g,\max}-\alpha_{f,\max}=\eee\]
by definition. Hence we can find $g$ arbitrarily close to $f$ with $\alpha_{g,\max}^*<\alpha_{g,\max}$ in this case.

Assume now that the only possible choices for $\omega_f$ are amongst $1^\infty$ and $0^\infty$. If $A$ can be chosen as in the first case, differing from the identically 1 and identically 0 strings of length $k$, then the previous argument might be repeated, thus it suffices to observe the cases when $\omega_f$ and $A$ can only be identically 1 or  identically 0. Clearly without loss of generality we can assume that the former one holds. 
In this case we perturb $f$ as follows: let $A = 11\ldots 10$, which is a block consisting of $k$-many $1$'s then followed by a $0$. We define $g\in\PCC^{k+1}(\Omega)$ such that $g = f+ \eee\ind_{[A]}$. Then $\alpha_{g,\max}-\alpha_{f,\max}=\eee$ as previously. 
Moreover, if $\omega'$ is periodic then we compute that
\begin{align*}\label{g-birk}
\lim_{N \to \infty} \frac{1}{N} \sum_{n=1}^N g(\sigma^n \omega') 
&= \lim_{N \to \infty} \frac{1}{N} \sum_{n=1}^N f(\sigma^n \omega') + \lim_{N \to \infty}\frac{\eee}{N} \sum_{n=1}^N \ind_{[A]}(\sigma^n \omega') \\
&\leq \alpha_{f, \max}^* + \eee \cdot  \lim_{N \to \infty}\frac{1}{N} \sum_{n=1}^N \ind_{[A]}(\sigma^n \omega') .
\end{align*}
Note that $ \lim_{N \to \infty}\frac{1}{N} \sum_{n=1}^N \ind_{[A]}(\sigma^n \omega') $, the relative frequency of $A$ in $\omega'$ is at most $\frac{1}{k+1}$ (which is obtained when $\omega' = A^\infty$). This implies that if $\omega_g$ is the maximal periodic configuration for $g$, then 
$$\alpha_{g, \max}^* = \lim_{N \to \infty}\frac{1}{N} \sum_{n=1}^N g(\sigma^n \omega_g) \leq \alpha_{f, \max}^* + \frac{\eee}{k+1} < \alpha_{f, \max}^* + \eee = \alpha_{g, \max}. $$

Thus in both of these two cases we showed that any $f\in\PCC^k(\Omega)$ can be approximated by functions satisfying $\alpha_{g,\max}^*<\alpha_{g,\max}$. It yields that such functions also form a dense set, which concludes the proof.
\end{proof}

\begin{remark}\label{revealed}
	In ergodic optimization, a function $f \in C(\Omega)$ for which $\alpha_{f, \max}^* = \alpha_{f, \max}$ is called \textit{revealed} (cf. \cite[\S5]{J17}). Theorem $\ref{*typsup}$ tells us that the set of revealed functions in $C(\Omega)$ forms a nowhere dense set.
\end{remark}

\section{One-sided derivatives of the Birkhoff spectra at endpoints}\label{Sonesderiv}

In this section 
for functions with continuous spectrum
we are interested in the one-sided derivatives of the spectrum at the endpoints of its support in the direction of the interior of the support.

\subsection
{One-sided derivatives at the endpoints of spectra for generic functions}
 \label{SS-GCID}
For the generic continuous function we have already seen in Theorem \ref{contspectrum} that the spectrum is continuous at these endpoints, and as in the direction of the exterior of $L_f$ the spectrum is constant zero, the one-sided derivative is also zero. 
On the other hand, towards the interior of the support it is of infinite absolute value as we see in the next theorem.

\begin{theorem} \label{typicalderivative}
For the generic continuous function $f\in C(\Omega)$, we have $\partial^- S_f(\alpha_{f,\max}^*)=-\infty$, while $\partial^+ S_f(\alpha_{f,\min}^*)=\infty$.
\end{theorem}

We start with a lemma which will be the building block for the proof of the above theorem.

\begin{lemma} \label{derivativelemma}
Let $f_0 \in C(\Omega)$, $\varepsilon > 0$, and $\nu \in \NN$ be given. 
Then there exists $f_2 \in C(\Omega)$ and $\delta > 0$ such that $\norm{f_0 - f_2} <\varepsilon/2$, $\delta <\varepsilon/2$, and for any 
$f \in B(f_2, \delta)\sse B(f_{0},\eee)$
there exists
$\aaa'<\aaa_{f,\max}^{*}$
such that 
\begin{equation}\label{*DL1a*}
\frac{S_{f}(\aaa')-S_{f}(\aaa_{f,\max}^{*})}{\aaa'-\aaa_{f,\max}^{*}}<-\nu.
\end{equation}
\end{lemma}

\begin{remark} \label{derivativeremark}
As $S_f$ is concave on the interval $L_f$, the inequality (\ref{*DL1a*}) in the lemma implies $\partial^{-}S_f(\alpha_{f,  \max}^*)<-\nu$. 
\end{remark}

\begin{proof}
Using Theorem \ref{*ppropdisc} choose $f_1\in PCC(\OOO)$ with $\|f_0-f_1\|<\eee/4$
such that $\eee_1=S_{f_1} (\aaa_{f_1,\max}^{*})>0$.

Set $\eee_2=\min\{ \eee_1,\frac{\eee}{2},1/2 \}$.

Using Theorem \ref{contspectrum} choose $f_2\in C(\OOO)$
such that $$\|f_1-f_2\|<\frac{\eee_2}{10\nu}\text{ and }S_{f_2}(\aaa_{f_2,\max}^{*})=0.$$
By Lemma \ref{Lemma:ptEstimate} and Remark \ref{*remsup} 
applied to $f_{1}$ and $f_{2}$ we obtain that
$\aaa_{f_2,\max}^{*}<\aaa_{f_1,\max}^{*}+\frac{\eee_2}{10\nu}$
and there exists $\aaa'>\aaa_{f_1,\max}^{*}-\frac{\eee_2}{10\nu}$  such that
\begin{equation}\label{*Sfk*}
S_{f_2}(\aaa')\geq S_{f_1}(\aaa_{f_1,\max}^{*})=\eee_1\geq \eee_2.
\end{equation} 
Then
\begin{equation}\label{*DL2*a}
\aaa_{f_2,\max}^{*}-\aaa'<2\cdot \frac{\eee_2}{10\nu}.
\end{equation}
Keep in mind that $S_{f_2}(\aaa_{f_2,\max}^{*})=0$ and choose $\ddd_1 >0$
such that 
\begin{equation}\label{*ept*}
S_{f_2}(\aaa)<\frac{\eee_2}{20}\text{ holds for }\aaa\in (\aaa_{f_2,\max}^{*}-\ddd_1 ,\aaa_{f_2,\max}^{*}].
\end{equation}
Observe that from \eqref{*Sfk*} it also follows that
$\aaa_{f_2,\min}^{*}\leq \aaa'<\aaa_{f_2,\max}^{*}-\ddd_{1}.$
Now choose $\ddd_2>0$  such that 
\begin{equation}\label{*ddchoice*}
\ddd_2<\min \Big \{ \frac{\aaa_{f_2,\max}^{*}-\aaa'}{10}, \frac{\ddd_1 }{5},\frac{\eee_2}{20\nu} \Big  \}.
\end{equation}
Using this $\ddd_{2}$ as $\eee$ in Theorem \ref{Thm:NormCont}
select $\ddd\in (0,\ddd_{2})$  such that for $f\in B(f_{2},\ddd)$
we have
\begin{equation}\label{*Sfkest*}
|S_{f}(\aaa)-S_{f_2}(\aaa)|<\ddd_{2} \text{ for }\aaa\in (\aaa_{f_2,\min}^{*}+\ddd_{2},\aaa_{f_2,\max}^{*}-\ddd_{2}).
\end{equation}

Suppose $f\in B(f_2,\ddd)$. 
Then by Lemma \ref{Lemma:ptEstimate},
 Remark \ref{*remsup}, \eqref{*DL2*a} and \eqref{*ddchoice*}
we obtain
$$|\aaa_{f,\max}^{*}-\aaa_{f_2,\max}^{*}|<\ddd_{2} \text{ and hence }|\aaa'-\aaa_{f,\max}^{*}|<1.1(\aaa_{f_{2},\max}^{*}-\aaa')< 1.1\cdot \frac{\eee_2}{5\nu}.$$

By \eqref{*ept*}, $S_{f_{2}}(\alpha_{f_2,\max}^*-\ddd_{1}/2)< \eee_{2}/20$ and then by \eqref{*Sfkest*},
$S_{f}(\alpha_{f_2,\max}^*-\ddd_{1}/2)< \eee_{2}/10<1$. By concavity of $S_{f}$ and $S_{f}(\int f)=1$
it is clear that $S_{f}$ is monotone decreasing on $[\aaa_{f,\max}^{*}-\ddd_{1}/2,
\aaa_{f,\max}^{*}]$ and hence 
\begin{equation}\label{*sfmaxest*}
S_{f}(\aaa_{f,\max}^{*})<\frac{\eee_2}{10}.
\end{equation}

Using  \eqref{*Sfk*}, \eqref{*ddchoice*} and  \eqref{*Sfkest*} we infer
$$S_{f}(\aaa')>S_{f_2}(\aaa')-\ddd_{2}\geq 0.9 \eee_2.$$
By this, \eqref{*sfmaxest*} and \eqref{*DL2*a}
$$\frac{S_{f}(\aaa')-S_{f}(\aaa_{f,\max}^{*})}{\aaa'-\aaa_{f,\max}^{*}}<-\frac{0.8\eee_2}{1.1\cdot \frac{\eee_2}{5\nu}}<-\nu.$$ 
\end{proof}


\begin{remark}\label{*rem412}
We   remark that  due to symmetry reasons a version of Lemma \ref{derivativelemma}
also holds at the other endpoint, $\aaa_{f,\min}^{*}$ of the spectrum yielding that for any
$f \in B(f_2, \delta)\sse B(f_{0},\eee)$
there exists
$\aaa'>\aaa_{f,\min}^{*}$
such that 
\begin{equation}\label{*DL1a*min}
\frac{S_{f}(\aaa')-S_{f}(\aaa_{f,\min}^{*})}{\aaa'-\aaa_{f,\min}^{*}}>\nu.
\end{equation}
\end{remark}

As we observed earlier in the one-dimensional case $S_{f}$ is continuous
on $[\aaa_{f,\min}^{*},\aaa_{f,\max}^{*}]$ hence even in case of discontinuous spectra
one can consider $\partial^{-}S_f(\alpha_{f,  \max}^*)$ and
$\partial^{+}S_f(\alpha_{f,  \max}^*)$, one might have a one-sided discontinuity
only in the direction pointing towards the exterior of the support of the spectrum.

Lemma \ref{derivativelemma} easily implies Theorem \ref{typicalderivative}:


\begin{proof}[Proof of Theorem $\ref{typicalderivative}$]
Consider an arbitrary $f_{0}\in C(\Omega)$ and $\varepsilon>0$. 
Fix $\nu\in\mathbb{N}$. 
We may apply Lemma \ref{derivativelemma} and Remark \ref{derivativeremark} to see that $B(f_{0},\varepsilon)$ contains a smaller open set $B(f_2,\delta)$ of $C(\Omega)$ such that for any $f\in B(f_2,\delta)$ we have $\partial^{-}S_f(\alpha_{f,  \max}^*) < -\nu$. 
It implies that the complement of
\[A_\nu=\{f\in C(\Omega) : \partial^{-}S_f(\alpha_{f,  \max}^*) < -\nu\}\]
is nowhere dense for any $\nu$. 
Hence $A=\bigcup_{\nu=1}^{\infty}A_\nu$ is a residual set of $C(\Omega)$, yielding that for the generic continuous function $f \in C(\Omega)$, we have $\partial^- S_f(\alpha_{f,\max}^*)=-\infty$.

However, by Remark \ref{*rem412} we may conclude the same way that for the generic continuous function $f \in C(\Omega)$, we have $\partial^+ S_f(\alpha_{f,\min}^*)=\infty$. 
Thus for the generic continuous function, we have both of these prescribed equalities, which concludes the proof.
\end{proof}

\subsection{Finite one-sided derivatives at the endpoints of the spectrum} \label{SS-CFFD}
Now our goal is to construct a continuous function $f$ with the property that the spectrum $S_f$ is continuous, but it is not generic in the above sense, that is the one-sided derivatives in the endpoints $\alpha_{f,\min}^*$ and $\alpha_{f,\max}^*$ are finite. 

\begin{theorem} \label{finitederivatives} 
There exists $f\in C_0(\Omega)$ such that $S_f$ is continuous, $\alpha_{f,\min}^*=-1$ and $\alpha_{f,\max}^*=1$, and $\partial^- S_f(\alpha_{f,\max}^*)>-\infty$, while $\partial^+ S_f(\alpha_{f,\min}^*)<\infty$. 
Moreover, these derivatives can be arbitrarily close to $-1$ and $1$, respectively.
\end{theorem}

The first step in this direction is the following lemma, in which we give upper bounds on a value of the spectrum for a suitably defined function.
Since $S_{f}(\int fd\lll)=1$ if we have a function with continuous spectrum then by concavity of the spectrum  
$\partial^- S_f(\alpha_{f,\max}^*)\leq  - 1/(\alpha_{f,\max}^*-\int fd\lll)  $ and $\partial^+S_f(\alpha_{f,\min}^*)\geq 1/(\int fd\lll-\alpha_{f,\min}^*)$.

In the next Lemma we define a $\PCC$ function with "very small" spectrum.
This type of functions serve as building blocks in the proof of Theorem \ref{finitederivatives}.

\begin{lemma} \label{spectrumbound}
Let $b>a$, and let $f:\Omega\to\mathbb{R}$ be such that $f(\omega)=b$ if the first $L$ coordinates of $\omega$ is 1, otherwise $f(\omega)=a$. 
Moreover, fix $\varepsilon>0$ and $0<\beta<1$. 
Then if $L$ is sufficiently large, then
\begin{equation} \label{spectrumboundeq}
S_f(t)\leq \beta+\varepsilon
\end{equation}
for $t=\beta a+(1-\beta)b$.
\end{lemma}

\begin{remark}\label{*remspectrbound}
See Figure \ref{fig1} for an illustration of this remark.
Observe that in the above lemma if $L$ is large then $\int fd\lll=b\cdot 2^{-L}+ a(1-2^{-L}) $ and hence $S_{f}(b\cdot 2^{-L}+ a\cdot (1-2^{-L}))=1$.
 The  point  $b\cdot 2^{-L}+ a\cdot (1-2^{-L})$ is very close to
$a=\aaa_{f,\min}$. 
It is also clear that $E_{f}(b)\not=\ess$, since $1^{\oo}$ belongs to it.
By also considering $0^{\oo}$ we see that $[a,b]=[\alpha_{f,\min}^*,\alpha_{f,\max}^*]$. 
Hence the line segment connecting $(b\cdot 2^{-L}+ a\cdot (1-2^{-L}),1)$ to
$(b,0)$ should be under the graph of $S_{f}$ on $[b\cdot 2^{-L}+ a\cdot (1-2^{-L}), b]$.
 If $\bbb$ is small then $t$ is very close to
$b$ and by concavity of the spectrum on $[b\cdot 2^{-L}+ a\cdot (1-2^{-L}),t]$
the graph of $S_{f}$ should be under the dashed line on the figure connecting
$(t,\bbb+\eee)=(\beta a+(1-\beta)b,\bbb+\eee)$ to $(b,0)$.
This implies that for small $\bbb$ and large $L$  apart from a very short interval near the endpoint $a$ the spectrum $S_{f}$ is very close
to the line segment connecting $(a,1)$ to $(b,0)$ and on $[a,b]$ approximates the upper
part of the boundary (shown with dotted line on the figure) of the right angled triangle with vertices $(a,0)$, $(a,1)$ and $(b,0)$. 
\end{remark}

\begin{figure}[h!]
\centering{
\resizebox{0.6\textwidth}{!}{\includegraphics{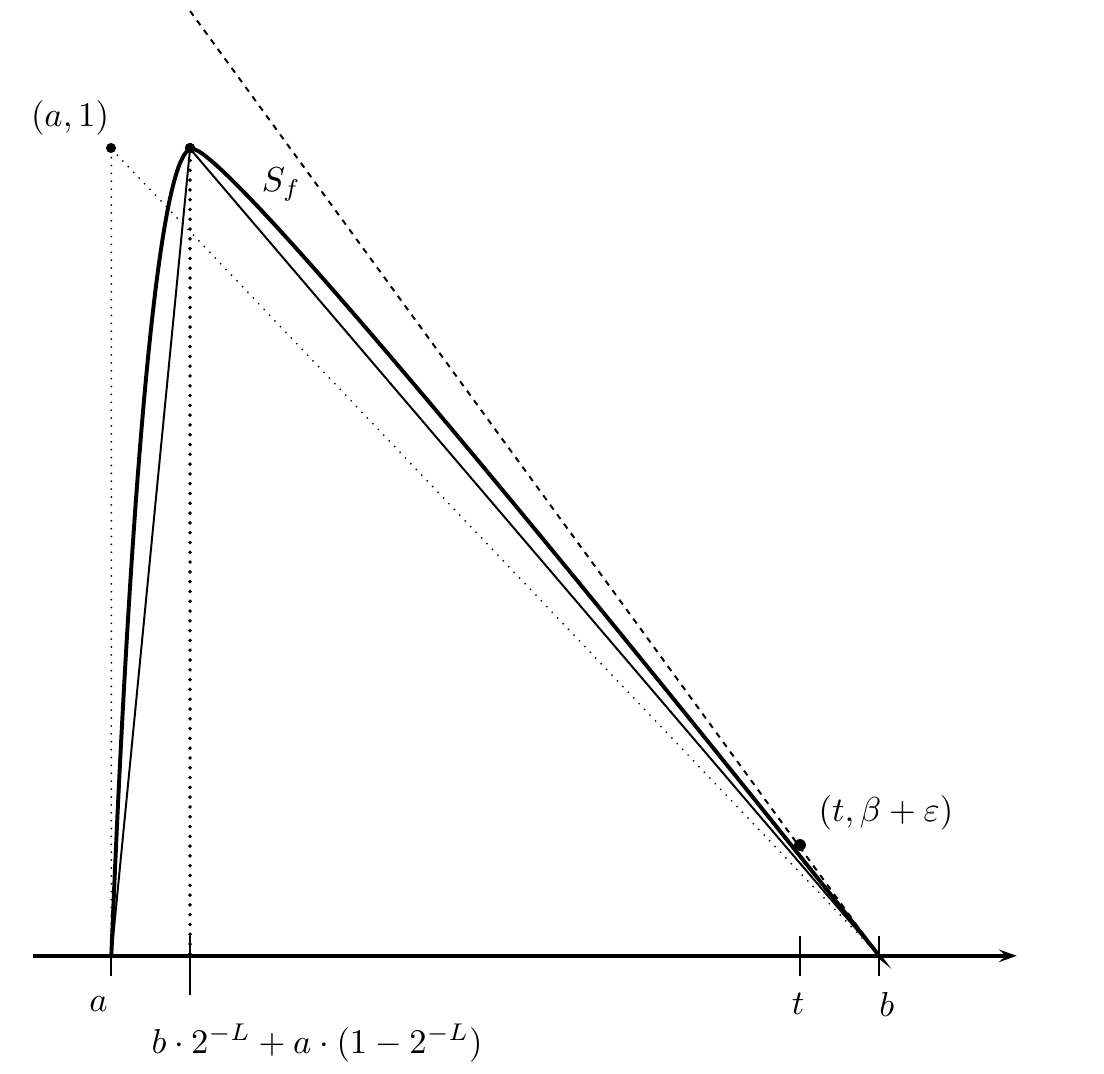}}}
\caption{An illustration of Remark \ref{*remspectrbound}.}
\label{fig1}
\end{figure}


\begin{proof}
Let $t=\beta a+(1-\beta)b$. 
Clearly it suffices to prove the statement of the lemma for small enough $\varepsilon$, thus we might assume that $\beta^*=\beta+\frac{\varepsilon}{2}<1$. 
We would like to estimate the dimension of
\[E_f(t)=\Big \{\omega:\lim_{N\to\infty}\frac{1}{N}\sum_{n=1}^{N}f(\sigma^n\omega)=t\Big \}.\]
This set contains $\omega$ if and only if it contains $\sigma(\omega)$, thus we can shift the sum by one for technical convenience. 
Moreover, if we replace the $\lim$ by a $\liminf$, we can deduce that this set is contained by
\[\Big \{\omega:\liminf_{N\to\infty}\frac{1}{N}\sum_{n=0}^{N-1}f(\sigma^n\omega)\geq t\Big \}.\]
If $\omega$ is in this set, then for large enough $N$ the corresponding ergodic average exceeds $t^*=\beta^* a+(1-\beta^*)b<t$, that is
\begin{equation} \label{setestimate}
E_f(t)\subset \bigcup_{m=1}^{\infty}\bigcap_{N=m}^{\infty}\left\{\ooo:\frac{1}{N}\sum_{n=0}^{N-1}f(\sigma^n\omega)\geq t^*\right\}.
\end{equation}
In the sequel for ease of notation we will use
$\left\{\frac{1}{N}\sum_{n=0}^{N-1}f(\sigma^n\omega)\geq t^*\right\}$
instead of $\left\{\ooo:\frac{1}{N}\sum_{n=0}^{N-1}f(\sigma^n\omega)\geq t^*\right\}.$
The union in \eqref{setestimate} is the union of a growing sequence of sets, thus the dimension is simply the limit of $\dim_H A_m$, where
\[A_m=\bigcap_{N=m}^{\infty}\left\{\frac{1}{N}\sum_{n=0}^{N-1}f(\sigma^n\omega)\geq t^*\right\}.\]
Now we focus on estimating the dimension of this set. 
To this end, we would like to count the cylinder sets of length $N+L-1$ which intersect
 $\left\{\frac{1}{N}\sum_{n=0}^{N-1}f(\sigma^n\omega)\geq t^*\right\}$ for large $N$, as they give a cover of $A_m$ for any $N\geq{m}$. (We are concerned with cylinders of length $N+L-1$ instead of the ones with length $N$ as the first $N+L-1$ coordinates affect $\sum_{n=0}^{N-1}f(\sigma^n\omega)$.) For our purposes it suffices to choose $N$ such that $L|N+L-1$, as we can diverge to infinity with $N$ even under this restriction and we need an upper estimate of the dimension.

The number of blocks consisting of at least $L$ consecutive 1s is at most $\frac{N+L-1}{L}$. 
If $L\geq{2}$, and there are $i$ such blocks, the number of ways to place them among the $N+L-1$ coordinates is at most $\binom{N+L-1}{2i}$, since the placement of each block can be uniquely specified by the coordinates for which the first and the last coordinates of the block occupy. (We note that it is indeed an upper estimate: this expression does not deal with the length of the blocks, neither with the fact that blocks are separated from each other with at least one intermediate coordinate.) Moreover, if $L\geq{5}$, then for the largest possible value of $i$, that is for $i=\frac{N+L-1}{L}$ we still have
\[2i=2\cdot\frac{N+L-1}{L}<\frac{N+L-1}{2}.\]
Thus the number of ways we can arrange the blocks of at least $L$ consecutive $1$s is at most
\begin{equation} \label{blockbound}
\sum_{i=0}^{\frac{N+L-1}{L}}\binom{N+L-1}{2i}\leq \left(\frac{N+L-1}{L}+1\right)\cdot\binom{N+L-1}{2\cdot\frac{N+L-1}{L}}
\end{equation}
$$\leq (N+L-1)\cdot\binom{N+L-1}{2\cdot\frac{N+L-1}{L}},$$
as the binomial coefficients are increasing until the middle ones.

We should also give a bound on the number of ways we can choose the other coordinates. 
Since $\frac{1}{N}\sum_{n=0}^{N-1}f(\sigma^n\omega)\geq t^*$, we know that most of the coordinates belong to one of the above blocks. 
More specifically, in the first $N$ coordinates there are at most $\beta^* N$ not covered by them, as otherwise the number of terms in $\sum_{n=0}^{N-1}f(\sigma^n\omega)$ with $f(\sigma^n\omega)=a$ exceeds $\beta^* N$, which yields that
\[\frac{1}{N}\sum_{n=0}^{N-1}f(\sigma^n\omega)< \beta^* a+(1-\beta^*)b=t^*.\]
Thus a raw upper estimate for the number of the ways we can choose the remaining coordinates in order to have an  $ N+L-1$-cylinder intersecting 
\[\left\{\frac{1}{N}\sum_{n=0}^{N-1}f(\sigma^n\omega)\geq t^*\right\}\]
is $2^{\beta^* N}\cdot 2^{L-1}$, where the last factor is simply the number of ways we can choose the last $L-1$ coordinates.

Combining the results of the preceding two paragraphs yields that 
\[\left\{\frac{1}{N}\sum_{n=0}^{N-1}f(\sigma^n\omega)\geq t^*\right\}\]
is covered by at most
\[(N+L-1)\cdot\binom{N+L-1}{2\cdot\frac{N+L-1}{L}}\cdot 2^{\beta^* N+L-1}\]
many cylinders of diameter $2^{-(N+L-1)}$. 
By using the standard $\binom{a}{b}\leq\left(\frac{ae}{b}\right)^b$ bound on the binomial coefficients, we can relax this upper bound to
\begin{equation} \label{boundcylinders}
(N+L-1)\cdot \left(\frac{eL}{2}\right)^{2\cdot\frac{N+L-1}{L}}\cdot 2^{\beta^* N+L-1}=k\cdot \left(\frac{eL}{2}\right)^{\frac{2k}{L}}\cdot 2^{\beta^* k}\cdot 2^{(1-\beta^*)(L-1)},
\end{equation}
where $k=N+L-1$. 
Notice that for large enough $L$ (and consequently, large enough $k$) we have
\[2^{\frac{\varepsilon}{2}}>\sqrt[k]{k}\left(\frac{eL}{2}\right)^{\frac{2}{L}},\]
as both factors on the right tend to $1$. Fix $L$ to be sufficiently large in order to guarantee this. 
Consequently, (\ref{boundcylinders}) can be estimated from above by
\begin{equation}
2^{(\beta^*+\frac{\varepsilon}{2})k}\cdot 2^{(1-\beta^*)(L-1)}.
\end{equation}
Hence 
\[\left\{\frac{1}{N}\sum_{n=0}^{N-1}f(\sigma^n\omega)\geq t^*\right\}\]
can be covered by at most $2^{(\beta^*+\frac{\varepsilon}{2})k}\cdot 2^{(1-\beta^*)(L-1)}$ many cylinders of diameter $2^{-k}$ for any $k$ with $L|k$. 
It immediately yields
\[\mathcal{H}_{2^{-k}}^{\beta^*+\frac{\varepsilon}{2}}\left(\left\{\frac{1}{N}\sum_{n=0}^{N-1}f(\sigma^n\omega)\geq t^*\right\}\right)\leq 2^{(1-\beta^*)(L-1)}\]
where $N=k-L+1$ as before. 
However, this set contains $A_m$ for large enough $k,N$, thus
\[\mathcal{H}_{2^{-k}}^{\beta^*+\frac{\varepsilon}{2}}\left(A_m \right)\leq 2^{(1-\beta^*)(L-1)}.\]
As $k,N$ can be arbitrarily large, it shows that in fact
\[\mathcal{H}^{\beta^*+\frac{\varepsilon}{2}}\left(A_m \right)\leq 2^{(1-\beta^*)(L-1)}\]
and consequently,
\[\dim_H(A_m)\leq \beta^*+\frac{\varepsilon}{2}=\beta+\varepsilon.\]
Consequently, by our initial observations
\[S_f(t)\leq \beta+\varepsilon,\]
as stated.
\end{proof}

\begin{proof}[Proof of Theorem $\ref{finitederivatives}$]
We define $f$ to be a more elaborate variant of the function appearing in Lemma \ref{spectrumbound}. 
Set $t_j=1-2^{-j}$.
Then $t_j\in (0,1)$ and $t_j\to 1$. We will define a strictly increasing sequence $(L_j)$ of positive integers, to be fixed later and chosen recursively. 
We can suppose that $L_{1}>5$.

Now we let $f(\omega)=t_j$ if $\omega$ starts with a block of $1$s of length at least $L_j$, but less than $L_{j+1}$. 
Moreover, $f(\omega)=-t_j$ if $\omega$ starts with a block of $0$s of length at least $L_j$, but less than $L_{j+1}$. 
Finally, let $f(1^{\oo})=1$ and $f(0^{\oo})=-1$ for the constant sequences, and let $f(\omega)=0$ for any remaining $\omega$. 
Due to symmetry, it is clear that $\int f=0$, and it is straightforward to check continuity. 
It remains to prove that the relevant derivatives are finite. 
By symmetry again, it suffices to verify $\partial^- S_f(\alpha_{f,\max}^*)>-\infty$. 
To this end, we will use an argument similar to the one seen in the proof of Lemma \ref{spectrumbound}. 

As in \eqref{setestimate}, we can deduce
\[E_f(t_{j+1})\subset \bigcup_{m=1}^{\infty}\bigcap_{N=m}^{\infty}\left\{\frac{1}{N}\sum_{n=0}^{N-1}f(\sigma^n\omega)\geq t_{j}\right\}.\]
This union is the union of a growing sequence of sets, thus the dimension is simply the limit of $\dim_H A_m$, where
\[A_m=\bigcap_{N=m}^{\infty}\left\{\frac{1}{N}\sum_{n=0}^{N-1}f(\sigma^n\omega)\geq t_{j}\right\}.\]
In order to estimate this dimension, we first introduce an auxiliary function, which is easier to examine. 
Explicitly, we let $f_j=0$, if $f\leq 0$, and we let $f_j=1$ if $f\geq t_{j}$. 
In any other case we let $f_j=f$. 
Then $f_j\geq{f}$, consequently
\[A_{m,j}=\bigcap_{N=m}^{\infty}\left\{\frac{1}{N}\sum_{n=0}^{N-1}f_j(\sigma^n\omega)\geq t_{j}\right\}\]
contains $A_m$. 
Thus it suffices to estimate the dimension of $A_{m,j}$. 
The argument is similar to the one in the proof of Lemma \ref{spectrumbound}. We would like to count the cylinder sets of length $N+L_j-1$ 
which  intersect $\left\{\frac{1}{N}\sum_{n=0}^{N-1}f_j(\sigma^n\omega)\geq t_j\right\}$ for large $N$, as they give a cover of $A_{m,j}$ for any $N\geq{m}$. 
In order to avoid the inconvenience caused by integer parts, we will only consider $N$s with certain divisibility properties, as before.

First of all, the number of blocks consisting of at least $L_j$ consecutive $1$s is at most $\frac{N+L_j-1}{L_j}$, which is an integer for infinitely many $N$. 
Thus the number of ways we can arrange the blocks of at least $L_j$ consecutive $1$s is at most
\begin{equation} \label{jblockbound}
\sum_{i=0}^{\frac{N+L_j-1}{L_j}}\binom{N+L_j-1}{2i}\leq \left(\frac{N+L_j-1}{L_j}+1\right)\cdot\binom{N+L_j-1}{2\cdot\frac{N+L_j-1}{L_j}}
\end{equation}
$$\leq (N+L_j-1)\cdot\binom{N+L_j-1}{2\cdot\frac{N+L_j-1}{L_j}},$$
using $L_j\geq L_{1}>{5}$, as in (\ref{blockbound}). We call these blocks $j$-blocks.

The novelty of cylinder counting in this proof compared to the previous one is that we have to take into account the blocks responsible for the values of $f_j$ between $0$ and $t_{j-1}$. As $\frac{1}{N}\sum_{n=0}^{N-1}f_j(\sigma^n\omega)\geq t_j$, in the first $N$ coordinates there are at most $\frac{1-t_j}{1-t_{j-1}}N=\frac{N}{2}$ not covered by the $j$-blocks, as otherwise the number of terms in $\sum_{n=0}^{N-1}f(\sigma^n\omega)$ with $f(\sigma^n\omega)\leq t_{j-1}$ is too large and we have
$\frac{1}{N}\sum_{n=0}^{N-1}f(\sigma^n\omega)< t_j.$
Thus beside the already placed $j$-blocks, there are at most $\frac{1-t_j}{1-t_{j-1}}N+L_j-1=\frac{N}{2}+L_j-1$ coordinates remaining, which might contain some $(j-1)$-blocks of at least $L_{j-1}$ consecutive $1$s. 
By a similar estimate to  (\ref{jblockbound}) we find that the number of possible arrangements of these $(j-1)$-blocks is at most
\begin{equation} \label{j-1blockbound}
\sum_{i=0}^{\frac{\frac{N}{2}+L_j-1}{L_{j-1}}}\binom{\frac{N}{2}+L_j-1}{2i}\leq \left(\frac{\frac{N}{2}+L_j-1}{L_{j-1}}+1\right)\cdot\binom{\frac{N}{2}+L_j-1}{2\cdot\frac{\frac{N}{2}+L_j-1}{L_{j-1}}}
\end{equation}
$$\leq \Big ({\frac{N}{2}+L_j-1}\Big )\cdot\binom{\frac{N}{2}+L_j-1}{2\cdot\frac{\frac{N}{2}+L_j-1}{L_{j-1}}},$$
using $L_{j-1}\geq L_{1}>5$.

Suppose that $j_{0}\in \{ 0,...,j-1 \}.$
Proceeding recursively, by the same argument we can conclude that the union of the $(j-i)$-blocks taken for $i=0,1,...,j_0-1$ cover all but at most $\frac{1-t_{j}}{1-t_{j_0}}N=\frac{N}{2^{j_0}}$ of the first $N$ coordinates. 
Thus beside these blocks there are at most $\frac{N}{2^{j_0}}+L_j-1$ coordinates remaining, which yields similarly to (\ref{j-1blockbound}) that the number of possible arrangements of the $(j-j_0)$-blocks is at most
\begin{equation} \label{blockboundgeneral}
\Big ({\frac{N}{2^{j_0}}+L_j-1}\Big )\cdot\binom{\frac{N}{2^{j_0}}+L_j-1}{2\cdot\frac{\frac{N}{2^{j_0}}+L_j-1}{L_{j-j_0}}}<(N+L_j-1)\cdot\binom{\frac{N}{2^{j_0}}+L_j-1}{2\cdot\frac{\frac{N}{2^{j_0}}+L_j-1}{L_{j-j_0}}}.
\end{equation}
We can use this bound for $j_0=0,1,...,j-1$. (We note that for infinitely many values of $N$ each number appearing in the above binomial coefficients is an integer.) Finally, there can be coordinates which are not contained by any such block. 
At most $(1-t_j)N$ of them in the first $N$ coordinates, and arbitrarily many of them in the last $L_j-1$ coordinates. 
Thus they can be chosen at most $2^{(1-t_j)N+L_j-1}$ different ways. 
Hence the number of cylinders which intersect $\left\{\frac{1}{N}\sum_{n=0}^{N-1}f_j(\sigma^n\omega)\geq t_j\right\}$ can be bounded by taking the product of the estimates in (\ref{blockboundgeneral}), and multiplying it by $2^{(1-t_j)N+L_{j}-1}$. 
Hence $\left\{\frac{1}{N}\sum_{n=0}^{N-1}f_j(\sigma^n\omega)\geq t_j\right\}$ can be covered by at most
\begin{equation} \label{thmboundcylinders}
(N+L_j-1)^j\cdot 2^{(1-t_j)N+L_j-1}\cdot \prod_{j_0=0}^{j-1}\binom{\frac{N}{2^{j_0}}+L_j-1}{2\cdot\frac{\frac{N}{2^{j_0}}+L_j-1}{L_{j-j_0}}}
\end{equation}
many cylinders of diameter $2^{-(N+L_j-1)}$. 
Observe that the $j_{0}=0$ case in \eqref{thmboundcylinders}
includes the estimate \eqref{jblockbound}.
By the standard estimate of binomial coefficients we can estimate it further from above by
\begin{equation} \label{thmboundcylinders2a}
(N+L_j-1)^j\cdot 2^{(1-t_j)N+L_j-1}\prod_{j_0=0}^{j-1}\left(\frac{eL_{j-j_0}}{2}\right)^{2\cdot\frac{\frac{N}{2^{j_0}}+L_j-1}{L_{j-j_0}}}.
\end{equation}
Introduce the notation $k=N+L_j-1$ again. 
By factoring out constants depending on $L_1,...,L_{j}$ into a constant denoted by $C(L_1,...,L_{j})$,
and rearranging \eqref{thmboundcylinders2a} one can obtain that it equals 
\begin{equation} \label{thmboundcylinders2}
C(L_1,...,L_{j})\cdot k^j\cdot 2^{(1-t_j)k}\prod_{j_0=0}^{j-1}\left(\frac{eL_{j-j_0}}{2}\right)^{\frac{2k}{2^{j_0}L_{j-j_0}}}.
\end{equation}
This formulation leads us to a suitable choice of $L_n$: for an arbitrary fixed $\tau>0$, define $L_n$ large enough to guarantee that
\begin{equation} \label{Lchoice}
\left(\frac{eL_{n}}{2}\right)^{\frac{2}{L_{n}}}<2^{{\tau}/{2^{2n}}}.
\end{equation}
With this choice, (\ref{thmboundcylinders2}) can be estimated by
\begin{equation} \label{thmboundcylinders3}
C(L_1,...,L_{j})\cdot k^j\cdot 2^{(1-t_j)k}\prod_{j_0=0}^{j-1}2^{{\tau k}/{2^{2j-j_0}}}\leq C(L_1,...,L_{j})\cdot k^j\cdot 2^{(1-t_j+\frac{\tau}{2^j})k}
\end{equation}
$$\leq C(L_1,...,L_{j})\cdot 2^{(1-t_j+\frac{2\tau}{2^j})k},$$
where the last inequality holds for large enough $N,k$. 
It immediately yields
\[\mathcal{H}_{2^{-k}}^{1-t_j+\frac{2\tau}{2^j}}\left(\left\{\frac{1}{N}\sum_{n=0}^{N-1}f_j(\sigma^n\omega)\geq t_j\right\}\right)\leq C(L_1,...,L_j)\]
where $N=k-L_{j}+1$ as before. 
However, this set contains $A_{m,j}$ for large enough $k,N$, thus
\[\mathcal{H}_{2^{-k}}^{1-t_j+\frac{2\tau}{2^j}}\left(A_m \right)\leq C(L_1,...,L_j).\]
As $k,N$ can be arbitrarily large, it shows that in fact
\[\mathcal{H}^{1-t_j+\frac{2\tau}{2^j}}\left(A_m \right)\leq C(L_1,...,L_j)\]
and consequently,
\[\dim_H(A_{m,j})\leq 1-t_j+\frac{2\tau}{2^j}.\]
Consequently, by our initial observations
\[S_f(t_j)\leq 1-t_j+\frac{2\tau}{2^j},\]
that is, using $t_{j}=1-2^{-j}$ we  have
\[S_f(1-2^{-j})\leq \frac{1+2\tau}{2^j}.\]
Thus if we calculate the left derivative of $S_{f}$ at $1$ by going along the sequence $t_j$, we find that it is at most $-(1+2\tau)>-\infty$, which concludes the proof.
\end{proof}

\begin{remark}\label{*remmin}
We note that as the spectrum is concave, for any function $f\in C_0(\Omega)$ such that $\alpha_{f,\min}^*=-1$ and $\alpha_{f,\max}^*=1$ we have that the graph of $S_f$ is above the triangle graph with vertices $(-1,0),(0,1),(1,0)$. 
On the other hand, it must be below the constant 1 function in the  interval $[-1,1]$. 
It is natural to ask whether these extremes can be attained/approximated. 
We do not give the complete answer for these questions, but make a few observations.

First of all, Theorem \ref{finitederivatives} easily yields that $S_f$ can be arbitrarily close to the triangle graph: notably for the function $f$ constructed in the previous proof, $S_f$ is contained by the triangle with vertices $(-1,0), (0,1+2\tau), (1,0)$ due to concavity. 
Thus the theoretic minimum can be approximated.

On the other hand, if we would like to construct some $f$ such that $S_f$ is considerably large, we can consider a function similar to the one in Example \ref{*propdisc}. More explicitly, let $f\in\PCC^{2k+1}(\Omega)$ be such that it takes the value -1 on cylinders which contain more $0$s than $1$s in their first $2k+1$ coordinates, and $f(\omega)=1$ otherwise. 
As in the proof of Example \ref{*propdisc}, we can show by Hutchinson's theorem that $S_f(-1)=S_f(1)$ is at least $\frac{k}{2k+1}$. 
Thus the piecewise linear graph determined by the vertices (-1,1/2), (0,1), (1,1/2) can be arbitrarily close to a lower estimate of the spectrum, which means that $S_f$ is considerably large, even though it is far from what we strove for.

We also provide another example, which displays that $S_f(\alpha_{f,\max}^*)$ can be arbitrarily close to 1 even for nonconstant functions, if we drop the condition that $\alpha_{f,\max}^*=1$. 
Notably, let $f\in\PCC^k(\Omega)$ such that it takes the value $-1$ if the first $k$ coordinates equal $0$, while it takes the value $\frac{1}{2^k-1}$ if these coordinates contain at least one $1$. Then similarly to the previous argument we have that $S_f\left(\frac{1}{2^k-1}\right)\geq \frac{k-1}{k}$. 
It would be interesting to see how large $S_f(\alpha_{f,\max}^*)$ can be if $f\in C_0(\Omega)$ such that $\alpha_{f,\min}^*=-1$ and $\alpha_{f,\max}^*=1$.

\end{remark}

\begin{remark}
It is natural to ask whether Theorem \ref{finitederivatives} holds if we restrict our attention to the class of H\"older functions or $\PCC$ functions. While we do not know whether there is a $\PCC$ function with finite one-sided derivatives
at the endpoints of the spectrum, the following theorem might make one believe that the answer to this question is negative:
\end{remark}

\begin{theorem}\label{pccinfiniteder}
	Assume that $f\in\PCC(\Omega)$ and $S_f$ is continuous. Then $\partial^- S_f(\alpha_{f,\max}^*)=-\infty$, while $\partial^+ S_f(\alpha_{f,\min}^*)=\infty$.
\end{theorem}

\begin{proof}
	Choose $k$ such that $f\in\PCC^k(\Omega)$. By symmetry, it clearly suffices to prove $\partial^- S_f(\alpha_{f,\max}^*)=-\infty$. Consider the directed graph $G=(V,E)$ defined in the proof of Lemma \ref{periodicmax}, and the set $\mathcal{C}$ of its cycles. By that reasoning it is clear that there exist cycles with distinct weight averages as otherwise for any infinite path $\Gamma$ we would get the same weight average in limit, which means that the ergodic averages have the same limit for all configurations, hence $S_f$ cannot be continuous. Moreover, as $G$ is connected as a directed graph, the graph of cycles $G_\mathcal{C}$ is also connected, in which the vertices are the elements of $\mathcal{C}$, and two of them are connected if they have a common vertex. 
	This, together with our previous observation implies that we can choose cycles $C$ and $C'$ such that they have a common vertex $v$, the cycle $C$ has maximal weight average amongst the elements of $\mathcal{C}$, while $C'$ does not. 
	Now consider the set of infinite paths in $G$ denoted by $H_{\beta}$ which consists of the paths which start from $v$, and can be partitioned into finite pieces $\Gamma_1,\Gamma_2,...$ such that each $\Gamma_i$ equals either $C$ or $C'$, and the density ${\mathbf d}\left(\{i:\Gamma_i=C\}\right)=\beta$. Then it is obvious to see that the weight average along any $\Gamma\in H_{\beta}$ tends to
	\[\beta\cdot\frac{1}{|C|}\sum_{e\in C}f(e)+(1-\beta)\cdot\frac{1}{|C'|}\sum_{e\in C'}f(e)=\beta\alpha_{f,\max}^*+(1-\beta)\alpha',\]
	where $\alpha'<\alpha_{f,\max}^*$ by the choice of $C'$. Thus if we take the corresponding configuration $ \omega(\Gamma)$, and in the ergodic averages we shift the indexing again by one, we see that
	\[\frac{1}{N}\sum_{n=0}^{N-1}f(\sigma^n\omega(\GGG))\to\beta\alpha_{f,\max}^*+(1-\beta)\alpha'.\]
	That is, if $\Omega_{\beta}$ denotes the set of $ \omega(\Gamma)$s for which $\Gamma\in H_{\beta}$, we have
	\begin{equation}\label{relationspectrum}
	\Omega_{\beta}\subseteq E_f(\beta\alpha_{f,\max}^*+(1-\beta)\alpha').
	\end{equation}
	However, the dimension of $\Omega_{\beta}$ is easy to estimate from below using the following mapping: for $ \omega(\Gamma)\in\Omega_{\beta}$ define $h( \omega(\Gamma))=h_1h_2...$ by
	\[h_i := \left\{ \begin{array}{ll} 
	1 & \text{if } \Gamma_i = C \\
	0 & \text{if } \Gamma_i = C'.
	\end{array}
	\right. \]
	Now $h$ is a H\"older-mapping. Note that the starting point of $\Gamma$ determines the first $k$ coordinates of $ \omega(\Gamma)$, and then going along $C$ (resp. $C'$) determines the next $|C|$ (resp. $|C'|$) coordinates. 
	By reversing this argument, if $K=\max\{|C|,|C'|\}$, the first $k+mK$ coordinates of $ \omega(\Gamma)$ uniquely determine the cycles $\Gamma_1,...,\Gamma_m$ in the decomposition of $\Gamma$. In other words, the first $m$ coordinates of $h( \omega(\Gamma))$ are uniquely determined by the first $k+mK$ coordinates of $ \omega(\Gamma)$. 
	From this, one easily obtains that $h$ is a H\"older-${1}/{K}$ mapping. 
	
	Moreover, by the definition of $H_\beta$ and $\Omega_\beta$, it is clear that $h(\Omega_\beta)$ equals the set of configurations in which the density of $1$s equals $\beta$. Thus by Example \ref{examplecontspect}, we can deduce that
	\[\dim_H(h(\Omega_\beta))=-\frac{\beta\log(\beta)+(1-\beta)\log(1-\beta)}{\log 2}.\]
	Hence as $h$ was H\"older-${1}/{K}$:
	\[\dim_H(\Omega_\beta)\geq-\frac{\beta\log(\beta)+(1-\beta)\log(1-\beta)}{K\log 2}.\]
	Thus by (\ref{relationspectrum}):
	\[S_f(\beta\alpha_{f,\max}^*+(1-\beta)\alpha')\geq-\frac{\beta\log(\beta)+(1-\beta)\log(1-\beta)}{K\log 2}.\]
	Consequently, also using that by continuity of $S_{f}$ we have $S_f(\alpha_{f,\max}^*)=0$ we infer
	\[\frac{S_f(\alpha_{f,\max}^*)-S_f(\beta\alpha_{f,\max}^*+(1-\beta)\alpha')}{\alpha_{f,\max}^*-(\beta\alpha_{f,\max}^*+(1-\beta)\alpha')}\leq \frac{\beta\log(\beta)+(1-\beta)\log(1-\beta)}{(1-\beta)(\alpha_{f,\max}^*-\alpha')K\log 2}\]
	However, the right hand side can be estimated from above by omitting the negative first term, and after simplifying by $1-\beta$ we see that it tends to $-\infty$ as $\beta\to 1$. Hence the same holds for the left hand side, showing that $\partial^- S_f(\alpha_{f,\max}^*)=-\infty$.
\end{proof}

\section*{Acknowledgments}
This project was initiated in October 2018, during the first author's visit to Facultad de Matem\'aticas at Pontificia Universidad Cat\'olica de Chile, which was partially supported by CONICYT PIA ACT 172001. The first author would like to thank the hospitality of PUC.

The second author is thankful to M\'at\'e Fellner for the valuable discussion.

We are also grateful to Thomas Jordan for his valuable comments 
given to the third author
and pointing out some important references.

Finally, we thank the anonymous referee for making comments which improved the presentation of our results.


\end{document}